\documentclass[a4paper,draft,12pt]{amsart}

\usepackage[T1]{fontenc}

\usepackage{amsthm, amssymb, amsmath}
\usepackage{mathrsfs, latexsym, color}
\usepackage[pdftex]{}

\numberwithin{equation}{section}

\theoremstyle{plain}
\newtheorem{theorem}{Theorem}[section]
\newtheorem{proposition}[theorem]{Proposition}
\newtheorem{lemma}[theorem]{Lemma}
\newtheorem{corollary}[theorem]{Corollary}
\newtheorem{claim}{Claim}

\theoremstyle{definition}

\theoremstyle{remark}

\renewcommand{\bar}{\overline}

\newcommand{\abs}[1]{\lvert#1\rvert}

\newcommand{\N}{\mathbb{N}}
\newcommand{\Z}{\mathbb{Z}}
\newcommand{\Q}{\mathbb{Q}}
\newcommand{\C}{\mathbb{C}}

\usepackage[all]{xy}
\usepackage{pgf}
\usepackage{tikz}
\usetikzlibrary{arrows,decorations.markings,shapes,decorations.pathreplacing}
\tikzstyle{vertex}=[circle, draw, fill=black, inner sep=0pt, minimum width=6pt]
\tikzstyle{vert}=[circle, draw=black, fill=white, inner sep=0pt, minimum width=6pt]
\tikzstyle{zc}=[circle, draw, fill=black, inner sep=0pt, minimum width=6pt]
\tikzstyle{pc}=[circle, draw=black, inner sep=1pt, minimum width=11pt, font=\tiny] 
\tikzstyle{nc}=[circle, draw=black, inner sep=1pt, minimum width=11pt, font=\tiny, style=dashed] 
\tikzstyle{pedge}=[draw,-]
\tikzstyle{wedge}=[draw,-,postaction={decorate}, decoration={markings,mark = at position 0.55 with {\arrow{stealth} } }]
\tikzstyle{dpedge}=[draw,-,postaction={decorate}]
\tikzstyle{wwedge}=[draw,-,postaction={decorate}, decoration={markings,mark = at position 0.5 with {\arrow{stealth}}, mark = at position 0.6 with {\arrow{stealth} } }]
\tikzstyle{wwedge2}=[draw,-,postaction={decorate}, decoration={markings,mark = at position 0.45 with {\arrow{stealth} }, mark = at position 0.65 with {\arrow{stealth} } }]
\tikzstyle{wwwedge}=[draw,-,postaction={decorate}, decoration={markings,mark = at position 0.65 with {\arrow{stealth} },mark = at position 0.45 with {\arrow{stealth} }, mark = at position 0.55 with {\arrow{stealth} } }]
\tikzstyle{wwwedge2}=[draw,-,postaction={decorate}, decoration={markings,mark = at position 0.7 with {\arrow{stealth} },mark = at position 0.4 with {\arrow{stealth} }, mark = at position 0.55 with {\arrow{stealth} } }]
\tikzstyle{wwwwedge}=[draw,-,postaction={decorate}, decoration={markings,mark = at position 0.7 with {\arrow{stealth} },mark = at position 0.4 with {\arrow{stealth} }, mark = at position 0.5 with {\arrow{stealth} }, mark = at position 0.6 with {\arrow{stealth} } }]
\tikzstyle{wwwnedge}=[draw,densely dashed,postaction={decorate}, decoration={markings,mark = at position 0.65 with {\arrow{stealth} },mark = at position 0.45 with {\arrow{stealth} }, mark = at position 0.55 with {\arrow{stealth} } }]
\tikzstyle{wwnedge}=[draw,densely dashed,postaction={decorate}, decoration={markings,mark = at position 0.5 with {\arrow{stealth} }, mark = at position 0.6 with {\arrow{stealth} } }]
\tikzstyle{wnedge}=[draw,densely dashed,postaction={decorate}, decoration={markings,mark = at position 0.55 with {\arrow{stealth} } }]
\tikzstyle{wnedge2}=[draw,densely dashed,postaction={decorate}, decoration={markings,mark = at position 0.6 with {\arrow{stealth} } }]
\tikzstyle{dnedge}=[draw,densely dashed,postaction={decorate}]
\tikzstyle{nedge}=[draw,densely dashed]
\tikzstyle{weight2}= [draw=white, fill=white, font=\scriptsize]
\tikzstyle{weight}= [font=\scriptsize]
\tikzstyle{empty}=[circle, draw=white, inner sep=2pt, fill=white, minimum width=4pt]
\tikzstyle{ghost}=[circle, draw=black, inner sep=1pt, style=densely dashed, minimum width=6pt, font=\tiny]
\tikzstyle{ghostc}=[circle, draw=black, inner sep=1pt, style=densely dashed, minimum width=10pt, font=\tiny]
\tikzstyle{dedge}=[draw,very thick,dotted]
\usepackage{float}
\subjclass[2010]{Primary }
\usepackage{fullpage}

\title[Lehmer's conjecture for Hermitian matrices over $\Z[i{]}$ and $\Z[\omega{]}$]{ Lehmer's conjecture for Hermitian matrices \\ over the Eisenstein and Gaussian integers  }

\date{}
\author{Gary Greaves}
\address{Mathematics Department, Royal Holloway, Egham, Surrey, TW20 0EX, UK.}
 \email{g.greaves@rhul.ac.uk}

\author{Graeme Taylor}
\address{Heilbronn Institute for Mathematical Research, University of Bristol, School of Mathematics, Howard House, Queens Avenue, Bristol BS8 1SN, UK.}
\email{magdt@bristol.ac.uk}

\begin{document}
	
	\begin{abstract}
			We solve Lehmer's problem for a class of polynomials arising from Hermitian matrices over the Eisenstein and Gaussian integers, that is, we show that all such polynomials have Mahler measure at least Lehmer's number $\tau_0 = 1.17628\dots$.
	\end{abstract}
	
	\maketitle
		
	\section{Introduction}
	
	Let $f(z) = (z-\alpha_1)\dots(z-\alpha_n)$ be a monic integer polynomial of degree $n$.
	The \textbf{Mahler measure} \cite{Mahler:Measure1962} $M(f)$ of $f$ is defined as the product of the absolute value of the zeros of $f$ that lie outside the unit circle. 
	In symbols
	\[
		M(f) = \prod_{j=1}^n \max (1, \abs{\alpha_j}),
	\]
	and it is clear that $M(f) \geqslant 1$.
	Lehmer's question \cite{Lehmer:33Cyclo} asks, for any monic integer polynomial $f$, what is the smallest $M(f)$ such that $M(f) > 1$.
	Along with his question, Lehmer noted the polynomial
	\[
		L(z) = z^{10} + z^9 - z^7 - z^6 - z^5 - z^4 - z^3 + z + 1,
	\]
	whose larger zero is $\tau_0 = 1.17628\dots$, and moreover $M(L) = \tau_0$. 
	The conjecture is that $\tau_0$ is the smallest such Mahler measure, and, although it is not clear that Lehmer himself made the conjecture, it is nevertheless known as \textbf{Lehmer's conjecture}.
	The survey article by Smyth \cite{Smyth:MMsurvey08} gives an overview of the history of Lehmer's conjecture.
	
	Let $A$ be a Hermitian matrix over an imaginary quadratic integer ring $R$.
	We take $\chi_A(x) = \det(xI - A)$ to be the \textbf{characteristic polynomial} of $A$.
	The characteristic polynomial of $A$ is monic and it has integer coefficients.
	Following McKee and Smyth, given $A$ we define the \textbf{associated reciprocal polynomial} $R_A(z) = \chi_A(z+1/z)$.
	McKee and Smyth \cite{McKee:noncycISM09} proved Lehmer's conjecture to be true for the polynomials $R_A(z)$ with $A$ is an integer symmetric matrix.
	The second author \cite{GTay:Lehmer11} has shown the conjecture to be true for the polynomials $R_A(z)$ with $A$ is a Hermitian matrix over imaginary quadratic rings $\mathcal O_{\Q(\sqrt{-d})}$ for $d = 2$ and $d > 3$.
	In this paper, we extend this result by treating the remaining imaginary quadratic integer rings, namely, the Eisenstein and Gaussian integers. 
	\begin{theorem}\label{thm:mainresult}
		Let $A$ be a Hermitian matrix over the Eisenstein or Gaussian integers.
		Then either $M(R_A) = 1$ or $M(R_A) \geqslant \tau_0$.
	\end{theorem}
	
	Smyth~\cite{Smyth:nonrecip71} showed that the smallest Mahler measure of a non-reciprocal monic integer polynomial is $1.32471\dots$.
	Hence, with regard to Lehmer's problem, one is only concerned with reciprocal polynomials.
	In his thesis~\cite{Greaves:thesis12}, the first author has shown that Lehmer's conjecture holds for polynomials $R_A$ where $A$ is a Hermitian matrix over a real quadratic integer ring.
	Hence, it remains to investigate Lehmer's conjecture for monic reciprocal integer polynomials that cannot be expressed as $R_A$ for some Hermitian matrix $A$ over a quadratic integer ring.

	\section{Preliminaries}
	
	\subsection{Equivalence and indecomposability}

	Write $M_n(R)$ for the ring of $n \times n$ matrices over a ring $R \subseteq \C$. 
	Let $U_n(R)$ denote the unitary group of matrices $Q$ in $M_n(R)$ which satisfy $QQ^* = Q^*Q = I$, where $Q^*$ denotes the Hermitian transpose of $Q$.
	Conjugation of a matrix $M \in M_n(R)$ by a matrix in $U_n(R)$ preserves the eigenvalues of $M$ and the base ring $R$.
	
	We can view a Hermitian matrix $A \in M_n(R)$ as a graph in the following way.
	Define an \textbf{$R$-graph} to be a weighted directed graph with weight function $w : V \times V \to R$ where $V$ is the set of vertices of the graph.
	Let $G$ be the $R$-graph corresponding to $A$ whose vertices are the rows of $A$ and the weight function $w$ is given by $w(u,v) = A_{uv}$ for $u,v \in V$.
	The vertices $u, v$ are adjacent if $w(u,v)$ is nonzero, and $w(u,v)$ is called the \textbf{edge-weight} of the edge $uv$.
	A vertex $v$ is called \textbf{charged} with charge $w(v,v)$ if $w(v,v)$ is nonzero, otherwise $v$ is called \textbf{uncharged} and a graph without charged vertices is called \textbf{uncharged} otherwise it is called \textbf{charged}.
	By a subgraph $H$ of $G$, we mean an induced subgraph and we say that $G$ \textbf{contains} $H$ and that $G$ is a \textbf{supergraph} of $H$.
	
	Now, for $R$ an imaginary integer ring, $U_n(R)$ is generated by permutation matrices and diagonal matrices of the form
	\[
		\operatorname{diag}(1,\dots,1,\mu,1,\dots,1),
	\]
	where $\mu$ is a unit in $R$.
	Let $D$ be such a diagonal matrix having $\mu$ in the $j$-th position.
	Conjugation by $D$ is called a $\mu$-\textbf{switching} at vertex $j$.
	This has the effect of multiplying all the out-neighbour edge-weights of $j$ by $\mu$ and all the in-neighbour edge-weights of $j$ by $\bar \mu$.
	The effect of conjugation by permutation matrices is just a relabelling of the vertices of the corresponding graph.
	Let $L$ be the Galois closure of the field generated by the elements of $R$ over $\Q$.
	The case of interest for us is when $R = \Z[i]$ or $R = \Z[\omega]$.
	We have the nice property that a Hermitian matrix $A$ over an imaginary quadratic integer ring $\mathcal O_{\Q(\sqrt{d})}$ has an integral characteristic polynomial.
	This follows from the fact that the nontrivial Galois automorphism $\sigma$ of $\Q(\sqrt{d})$ over $\Q$ is complex conjugation.
	Applying $\sigma$ to the coefficients of $\chi_A$ gives $\sigma(\chi_A(x)) = \det(xI - \sigma(A)) = \det(xI - A^\top) = \chi_A(x)$.
	Hence the coefficients of $\chi_A$ lie in the intersection of $\Q$ and $\mathcal O_{\Q(\sqrt{d})}$ which is $\Z$.
	Let $A$ and $B$ be two matrices in $M_n(R)$.
	We say that $A$ is \textbf{strongly equivalent} to $B$ if $A = \sigma(QBQ^*)$ for some $Q \in U_n(R)$ and some $\sigma \in \operatorname{Gal}(L/\Q)$, where $\sigma$ is applied componentwise to $QBQ^*$.
	The matrices $A$ and $B$ are merely called \textbf{equivalent} if $A$ is strongly equivalent to $\pm B$.
	The notions of equivalence and strong equivalence carry through to graphs in the natural way and, since all possible labellings of a graph are strongly equivalent, we do not need to label the vertices.
	
	Fundamental to our approach is the following result of Cauchy \cite{Hwang:Interlace04}; an even shorter proof is given by Fisk \cite{Fisk:Interlace05}.

	\begin{theorem}[Interlacing] \label{thm:interlacing} Let $A$ be an $n~\times~n$ Hermitian matrix with eigenvalues $\lambda_1 \leqslant \dots \leqslant \lambda_n$. Let $B$ be an $(n - 1)~\times~(n - 1)$ principal submatrix of $A$ with eigenvalues $\mu_1 \leqslant \dots \leqslant \mu_{n-1}$. Then the eigenvalues of $A$ and $B$ interlace. Namely,
		\[
			\lambda_1 \leqslant \mu_1 \leqslant \lambda_2 \leqslant \mu_2 \leqslant \dots \leqslant \mu_{n-1} \leqslant \lambda_n.
		\]
	\end{theorem}
	
	A matrix that is equivalent to a block diagonal matrix of more than one block is called \textbf{decomposable}, otherwise it is called \textbf{indecomposable}.
	The \textbf{underlying graph} of a Hermitian matrix $A$ is one whose vertices are the rows of $A$ and vertices $i$ and $j$ are adjacent if and only if $A_{ij}$ is nonzero. 
	A matrix is indecomposable if and only if its underlying graph is connected.
	The eigenvalues of a decomposable matrix are found by pooling together the eigenvalues of its blocks.
	
	A Hermitian matrix $A$ is called \textbf{cyclotomic} if the zeros of its characteristic polynomial $\chi_A$ are contained inside the interval $[-2,2]$.
	An indecomposable cyclotomic matrix that is not a principal submatrix of any other indecomposable cyclotomic matrix is called a \textbf{maximal} indecomposable cyclotomic matrix.
	The corresponding graph is called a maximal connected cyclotomic graph.
	Let $R$ be an imaginary quadratic integer ring.
	The authors \cite{GTay:cyclos10,Greaves:CycloEG11} have shown that for each connected cyclotomic $R$-graph there exists a maximal connected cyclotomic $R$-graph containing it.
	
	\subsection{Graphical interpretation}
	
	The language of graphs has already been alluded to in the previous section.
	We use graphs as a convenient way of representing an equivalence class of Hermitian matrices and we describe our drawing conventions below.
	In this paper we are only concerned with $R$-graphs where $R = \Z[i]$ and $R = \Z[\omega]$ where $\omega = 1/2 + \sqrt{-3}/2$.
	Vertices, if uncharged, are drawn as filled discs \tikz {\node[zc] {};}, if charged with positive charge $c$, vertices are drawn as $c$ with a solid circular outline \tikz {\node[pc] {$c$};}, and otherwise, if charged negatively with $-c$, vertices are drawn as $c$ with a dashed circular outline \tikz {\node[nc] {$c$};}.
	We use \tikz {\node[pc] {$\ast$};} to represent a vertex having any charge; all vertices are equivalent to \tikz {\node[pc] {$\ast$};}.
	Edges are drawn in accordance with Tables~\ref{tab:ziedge}~and~\ref{tab:zwedge} which describe edges for $\Z[i]$-graphs and $\Z[\omega]$-graphs respectively.
	For edges with a real edge-weight, the direction of the edge does not matter, and so to reduce clutter, we do not draw arrows for these edges. 
	For all other edges, the number of arrowheads reflects the norm of the edge-weight.
	\begin{table}[htbp]
		\begin{center}
		\begin{tabular}{c | c}
			Edge-weight & Visual representation \\
			\hline
			$1$ & \tikz { \path[pedge] (0,0) -- (1,0); } \\
			$-1$ & \tikz { \path[nedge] (0,0) -- (1,0); } \\
			$i$ & \tikz { \path[wedge] (0,0) -- (1,0); } \\
			$-i$ & \tikz { \path[wnedge] (0,0) -- (1,0); } \\
			$1 + i$ & \tikz { \path[wwedge] (0,0) -- (1,0); } \\
			$-1 - i$ & \tikz { \path[wwnedge] (0,0) -- (1,0); } \\
			$2$ & \tikz { \path[pedge] (0,0) -- node[weight2] {$2$} (1,0); }
		\end{tabular}
		\end{center}
		\caption{Edge drawing convention for $\Z[i]$-graphs.}
		\label{tab:ziedge}
	\end{table}
	\begin{table}[htbp]
		\begin{center}
		\begin{tabular}{c | c}
			Edge-weight & Visual representation \\
			\hline
			$1$ & \tikz { \path[pedge] (0,0) -- (1,0); } \\
			$-1$ & \tikz { \path[nedge] (0,0) -- (1,0); } \\
			$\omega$ & \tikz { \path[wedge] (0,0) -- (1,0); } \\
			$-\omega$ & \tikz { \path[wnedge] (0,0) -- (1,0); } \\
			$1 + \omega$ & \tikz { \path[wwwedge] (0,0) -- (1,0); } \\
			$-1 - \omega$ & \tikz { \path[wwwnedge] (0,0) -- (1,0); } \\
			$2$ & \tikz { \path[pedge] (0,0) -- node[weight2] {$2$} (1,0); }
		\end{tabular}
		\end{center}
		\caption{Edge drawing convention for $\Z[\omega]$-graphs.}
		\label{tab:zwedge}
	\end{table}
	The edge labelling \tikz { \path[pedge] (0,0) -- node[weight2] {$\ast$} (1,0); } represents an edge having any weight; all edges are equivalent to the edge \tikz { \path[pedge] (0,0) -- node[weight2] {$\ast$} (1,0); }.
	
	To ease language, by a graph we mean an $R$-graph for some undisclosed quadratic integer ring $R$.	
	
	\section{Cyclotomic matrices over the Eisenstein and Gaussian integers}
	
	Our main result in this paper uses the classification of cyclotomic matrices over the Eisenstein and Gaussian integers.
	We describe this classification below.
	
	\begin{lemma}\label{lem:cycclass}\cite{Greaves:CycloEG11}
		Up to equivalence, every indecomposable cyclotomic $\Z[i]$-matrix and every indecomposable cyclotomic $\Z[\omega]$-matrix is the adjacency matrix of a subgraph of one of the maximal graphs in Figures~\ref{fig:maxcycs1}, \ref{fig:maxcycs2}, \ref{fig:maxcycs3}, \ref{fig:maxcycs4}, \ref{fig:maxcycs5}, \ref{fig:maxcycs6}, \ref{fig:maxcycs8}, and \ref{fig:maxcycs10}.
	\end{lemma}
	
	\begin{figure}[htbp]
		\centering
			\begin{tikzpicture}[auto, scale=1.5]
				\begin{scope}
					\foreach \type/\pos/\name in {{vertex/(0,0)/a2}, {vertex/(0,1)/a1}, {vertex/(1,1)/b1}, {vertex/(1,0)/b2}, {vertex/(2,0)/e2}, {vertex/(2,1)/e1}, {empty/(2.6,1)/b11}, {empty/(2.6,0)/b21}, {empty/(2.4,0.6)/b12}, {empty/(2.4,0.4)/b22}, {empty/(3.4,1)/c11}, {empty/(3.4,0)/c21}, {empty/(3.6,0.6)/c12}, {empty/(3.6,0.4)/c22}, {vertex/(4,1)/c1}, {vertex/(4,0)/c2}, {vertex/(5,1)/d1}, {vertex/(5,0)/d2}, {vertex/(6,1)/f1}, {vertex/(6,0)/f2}}
						\node[\type] (\name) at \pos {};
					\foreach \pos/\name in {{(3,0.5)/\dots}, {(-.3,1)/A}, {(-0.3,0)/B}, {(6.3,1)/A}, {(6.3,0)/B}}
						\node at \pos {$\name$};
					\foreach \edgetype/\source/ \dest/\weight in {nedge/b1/a2/{}, pedge/a1/b1/{}, pedge/a1/b2/{}, nedge/a2/b2/{}, nedge/e1/b2/{}, pedge/b1/e1/{}, pedge/b1/e2/{}, nedge/b2/e2/{}, nedge/b21/e2/{}, pedge/e1/b11/{}, pedge/e1/b12/{}, nedge/e2/b22/{}, pedge/c11/c1/{}, nedge/c12/c1/{}, pedge/c22/c2/{}, nedge/c21/c2/{}, nedge/d1/c2/{}, pedge/c1/d1/{}, pedge/c1/d2/{}, nedge/c2/d2/{}, pedge/d1/f1/{}, nedge/d2/f2/{}}
						\path[\edgetype] (\source) -- node[weight] {$\weight$} (\dest);
				\end{scope}
				\begin{scope}
					\foreach \edgetype/\source/ \dest/\weight in {nedge/d2/f1/{}, pedge/d1/f2/{}}
						\path[\edgetype] (\source) -- node[weight] {$\weight$} (\dest);
				\end{scope}
			\end{tikzpicture}
			\begin{tikzpicture}[auto, scale=1.5]
				\begin{scope}
					\foreach \type/\pos/\name in {{vertex/(0,0)/a2}, {vertex/(0,1)/a1}, {vertex/(1,1)/b1}, {vertex/(1,0)/b2}, {vertex/(2,0)/e2}, {vertex/(2,1)/e1}, {empty/(2.6,1)/b11}, {empty/(2.6,0)/b21}, {empty/(2.4,0.6)/b12}, {empty/(2.4,0.4)/b22}, {empty/(3.4,1)/c11}, {empty/(3.4,0)/c21}, {empty/(3.6,0.6)/c12}, {empty/(3.6,0.4)/c22}, {vertex/(4,1)/c1}, {vertex/(4,0)/c2}, {vertex/(5,1)/d1}, {vertex/(5,0)/d2}, {vertex/(6,1)/f1}, {vertex/(6,0)/f2}}
						\node[\type] (\name) at \pos {};
					\foreach \pos/\name in {{(3,0.5)/\dots}, {(-.3,1)/A}, {(-0.3,0)/B}, {(6.3,1)/A}, {(6.3,0)/B}}
						\node at \pos {$\name$};
					\foreach \edgetype/\source/ \dest/\weight in {nedge/b1/a2/{}, pedge/a1/b1/{}, pedge/a1/b2/{}, nedge/a2/b2/{}, nedge/e1/b2/{}, pedge/b1/e1/{}, pedge/b1/e2/{}, nedge/b2/e2/{}, nedge/b21/e2/{}, pedge/e1/b11/{}, pedge/e1/b12/{}, nedge/e2/b22/{}, pedge/c11/c1/{}, nedge/c12/c1/{}, pedge/c22/c2/{}, nedge/c21/c2/{}, nedge/d1/c2/{}, pedge/c1/d1/{}, pedge/c1/d2/{}, nedge/c2/d2/{}, wedge/d1/f1/{}, wnedge/d2/f2/{}}
						\path[\edgetype] (\source) -- node[weight] {$\weight$} (\dest);
				\end{scope}
				\begin{scope}[decoration={markings,mark = at position 0.4 with {\arrow{stealth} } }]
					\foreach \edgetype/\source/ \dest/\weight in {dnedge/d2/f1/{}, dpedge/d1/f2/{}}
						\path[\edgetype] (\source) -- node[weight] {$\weight$} (\dest);
				\end{scope}
			\end{tikzpicture}
		\caption{The infinite families $T_{2k}$ and $T^{(x)}_{2k}$ (respectively) of $2k$-vertex maximal connected cyclotomic $\Z[x]$-graphs, for $k \geqslant 3$ and $x \in \{i, \omega \}$. (The two copies of vertices $A$ and $B$ should be identified to give a toral tessellation.) }
		\label{fig:maxcycs1}
	\end{figure}
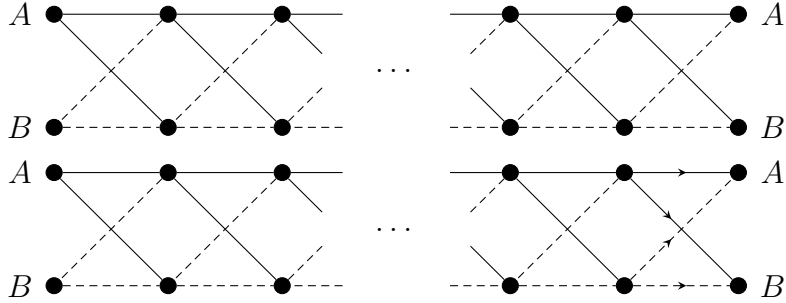

	\begin{figure}[htbp]
		\centering
			\begin{tikzpicture}[scale=1.5, auto]
				\begin{scope}
					\foreach \type/\pos/\name in {{vertex/(1,1)/b1}, {vertex/(1,0)/b2}, {vertex/(2,0)/e2}, {vertex/(2,1)/e1}, {empty/(2.6,1)/b11}, {empty/(2.6,0)/b21}, {empty/(2.4,0.6)/b12}, {empty/(2.4,0.4)/b22}, {empty/(3.4,1)/c11}, {empty/(3.4,0)/c21}, {empty/(3.6,0.6)/c12}, {empty/(3.6,0.4)/c22}, {vertex/(4,1)/c1}, {vertex/(4,0)/c2}, {vertex/(5,1)/d1}, {vertex/(5,0)/d2}}
						\node[\type] (\name) at \pos {};
					\foreach \type/\pos/\name in {{vertex/(0,0.5)/bgn}, {vertex/(6,0.5)/end}}
						\node[\type] (\name) at \pos {};
					\foreach \pos/\name in {{(3,0.5)/\dots}}
						\node at \pos {\name};
					\foreach \edgetype/\source/ \dest in {nedge/e1/b2, pedge/b1/e1, pedge/b1/e2, nedge/b2/e2, nedge/b21/e2, pedge/e1/b11, pedge/e1/b12, nedge/e2/b22, pedge/c11/c1, nedge/c12/c1, pedge/c22/c2, nedge/c21/c2, nedge/d1/c2, pedge/c1/d1, pedge/c1/d2, nedge/c2/d2}
						\path[\edgetype] (\source) -- (\dest);
					\foreach \edgetype/\source/\dest in {wwedge/b1/bgn, wwedge/b2/bgn, wwnedge/d2/end, wwedge/d1/end}
						\path[\edgetype] (\source) -- (\dest);
				\end{scope}
			\end{tikzpicture}
		\caption{The infinite family of $2k$-vertex maximal connected cyclotomic $\Z[i]$-graphs $C_{2k}$ for $k \geqslant 2$.}
		\label{fig:maxcycs2}
	\end{figure}
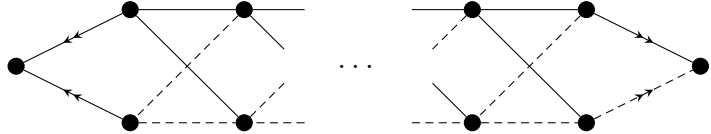

	\begin{figure}[htbp]
		\centering
			\begin{tikzpicture}[scale=1.5, auto]
					\foreach \type/\pos/\name in {{vertex/(1,1)/b1}, {vertex/(1,0)/b2}, {vertex/(2,0)/e2}, {vertex/(2,1)/e1}, {empty/(2.6,1)/b11}, {empty/(2.6,0)/b21}, {empty/(2.4,0.6)/b12}, {empty/(2.4,0.4)/b22}, {empty/(3.4,1)/c11}, {empty/(3.4,0)/c21}, {empty/(3.6,0.6)/c12}, {empty/(3.6,0.4)/c22}, {vertex/(4,1)/c1}, {vertex/(4,0)/c2}, {vertex/(5,1)/d1}, {vertex/(5,0)/d2}}
						\node[\type] (\name) at \pos {};
					\foreach \type/\pos/\name in {{pc/(0,0)/a2}, {pc/(0,1)/a1}, {pc/(6,1)/f1}, {pc/(6,0)/f2}}
						\node[\type] (\name) at \pos {$1$};
					\foreach \pos/\name in {{(3,0.5)/\dots}}
						\node at \pos {$\name$};
					\foreach \edgetype/\source/ \dest in {pedge/a1/a2, nedge/b1/a2, pedge/a1/b1, pedge/a1/b2, nedge/a2/b2, nedge/e1/b2, pedge/b1/e1, pedge/b1/e2, nedge/b2/e2, nedge/b21/e2, pedge/e1/b11, pedge/e1/b12, nedge/e2/b22, pedge/c11/c1, nedge/c12/c1, pedge/c22/c2, nedge/c21/c2, nedge/d1/c2, pedge/c1/d1, pedge/c1/d2, nedge/c2/d2, nedge/f1/d2, pedge/d1/f1, pedge/d1/f2, nedge/d2/f2, nedge/f1/f2}
						\path[\edgetype] (\source) -- (\dest);
					\node at (3,-0.2) {};
				\end{tikzpicture}
				\begin{tikzpicture}[scale=1.5, auto]
					\foreach \type/\pos/\name in {{vertex/(1,1)/b1}, {vertex/(1,0)/b2}, {vertex/(2,0)/e2}, {vertex/(2,1)/e1}, {empty/(2.6,1)/b11}, {empty/(2.6,0)/b21}, {empty/(2.4,0.6)/b12}, {empty/(2.4,0.4)/b22}, {empty/(3.4,1)/c11}, {empty/(3.4,0)/c21}, {empty/(3.6,0.6)/c12}, {empty/(3.6,0.4)/c22}, {vertex/(4,1)/c1}, {vertex/(4,0)/c2}, {vertex/(5,1)/d1}, {vertex/(5,0)/d2}}
						\node[\type] (\name) at \pos {};
					\foreach \type/\pos/\name in {{pc/(0,0)/a2}, {pc/(0,1)/a1}, {nc/(6,1)/f1}, {nc/(6,0)/f2}}
						\node[\type] (\name) at \pos {$1$};
					\foreach \pos/\name in {{(3,0.5)/\dots}}
						\node at \pos {$\name$};
					\foreach \edgetype/\source/ \dest in {pedge/a1/a2, nedge/b1/a2, pedge/a1/b1, pedge/a1/b2, nedge/a2/b2, nedge/e1/b2, pedge/b1/e1, pedge/b1/e2, nedge/b2/e2, nedge/b21/e2, pedge/e1/b11, pedge/e1/b12, nedge/e2/b22, pedge/c11/c1, nedge/c12/c1, pedge/c22/c2, nedge/c21/c2, nedge/d1/c2, pedge/c1/d1, pedge/c1/d2, nedge/c2/d2, nedge/f1/d2, pedge/d1/f1, pedge/d1/f2, nedge/d2/f2, pedge/f1/f2}
						\path[\edgetype] (\source) -- (\dest);
			\end{tikzpicture}
		\caption{The infinite families of $2k$-vertex maximal connected cyclotomic $\Z$-graphs $C_{2k}^{++}$ and $C_{2k}^{+-}$ for $k \geqslant 2$.}
		\label{fig:maxcycs3}
	\end{figure}
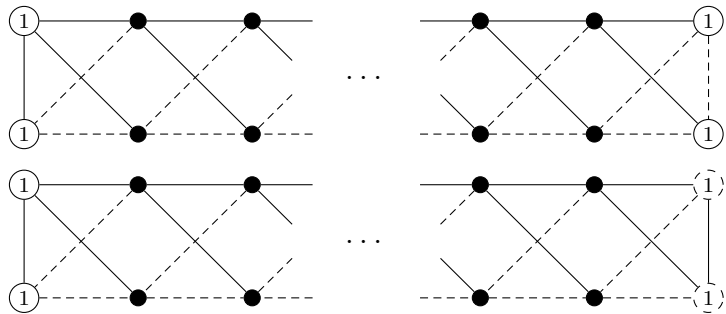

	\begin{figure}[htbp]
		\centering
			\begin{tikzpicture}[scale=1.5, auto]
				\begin{scope}
					\foreach \type/\pos/\name in {{vertex/(0,0.5)/bgn}, {vertex/(1,1)/b1}, {vertex/(1,0)/b2}, {vertex/(2,0)/e2}, {vertex/(2,1)/e1}, {empty/(2.6,1)/b11}, {empty/(2.6,0)/b21}, {empty/(2.4,0.6)/b12}, {empty/(2.4,0.4)/b22}, {empty/(3.4,1)/c11}, {empty/(3.4,0)/c21}, {empty/(3.6,0.6)/c12}, {empty/(3.6,0.4)/c22}, {vertex/(4,1)/c1}, {vertex/(4,0)/c2}, {vertex/(5,1)/d1}, {vertex/(5,0)/d2}}
						\node[\type] (\name) at \pos {};
					\foreach \type/\pos/\name in {{pc/(6,1)/f1}, {pc/(6,0)/f2}}
						\node[\type] (\name) at \pos {$1$};
					\foreach \pos/\name in {{(3,0.5)/\dots}}
						\node at \pos {$\name$};
					\foreach \edgetype/\source/ \dest in {nedge/e1/b2, pedge/b1/e1, pedge/b1/e2, nedge/b2/e2, nedge/b21/e2, pedge/e1/b11, pedge/e1/b12, nedge/e2/b22, pedge/c11/c1, nedge/c12/c1, pedge/c22/c2, nedge/c21/c2, nedge/d1/c2, pedge/c1/d1, pedge/c1/d2, nedge/c2/d2, nedge/f1/d2, pedge/d1/f1, pedge/d1/f2, nedge/d2/f2, nedge/f1/f2}
						\path[\edgetype] (\source) -- (\dest);
					\foreach \edgetype/\source/\dest in {wwedge/b1/bgn, wwedge/b2/bgn}
						\path[\edgetype] (\source) -- (\dest);
				\end{scope}
			\end{tikzpicture}
		\caption{The infinite family of $(2k+1)$-vertex maximal connected cyclotomic $\Z[i]$-graphs $C_{2k+1}$ for $k \geqslant 1$.}
		\label{fig:maxcycs4}
	\end{figure}
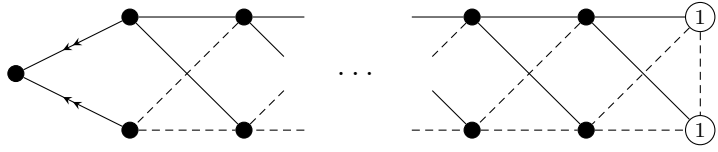
	
	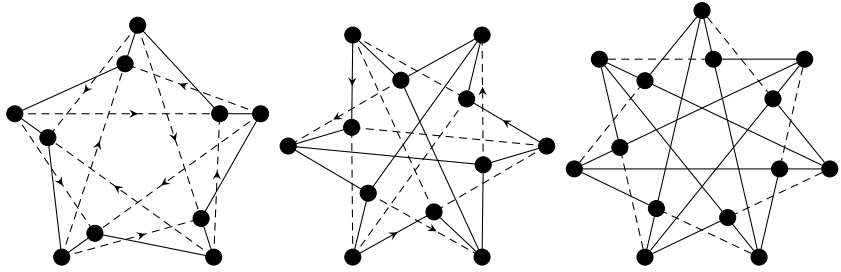
\begin{figure}[htbp]
		\centering
		\begin{tikzpicture}
			\newdimen\rad
			\rad=1.7cm
			\newdimen\radi
			\radi=1.2cm

			\foreach \x in {90,162,234,306,378}
			{
		    	\draw (\x:\rad) node[vertex] {};
				\draw (\x+8:\radi) node[vertex] {};
				\draw[pedge] (\x:\rad) -- (\x+8:\radi);
				\draw[pedge] (\x:\rad) -- (\x-72+8:\radi);
				\draw[wnedge2] (\x:\rad) -- (\x+72+8:\radi);
				\draw[wnedge2] (\x:\rad) -- (\x+216+8:\radi);
		    }
		\end{tikzpicture}
		\begin{tikzpicture}
			\newdimen\rad
			\rad=1.7cm
			\newdimen\radi
			\radi=0.9cm
			\def\shift{344}
			\foreach \x in {0,120,240}
			{
		    	\draw (\x:\rad) node[vertex] {};
				\draw (\x+\shift:\radi) node[vertex] {};
				\draw[pedge] (\x:\rad) -- (\x+\shift:\radi);
				\draw[nedge] (\x:\rad) -- (\x-60+\shift:\radi);
				\draw[nedge] (\x:\rad) -- (\x+180+\shift:\radi);
				\draw[wedge] (\x:\rad) -- (\x+60+\shift:\radi);
		    }
			\foreach \x in {60,180,300}
			{
		    	\draw (\x:\rad) node[vertex] {};
				\draw (\x+\shift:\radi) node[vertex] {};
				\draw[pedge] (\x:\rad) -- (\x+\shift:\radi);
				\draw[pedge] (\x:\rad) -- (\x+60+\shift:\radi);
				\draw[pedge] (\x:\rad) -- (\x+180+\shift:\radi);
				\draw[wnedge2] (\x-60+\shift:\radi) -- (\x:\rad);
		    }
		\end{tikzpicture}
			\begin{tikzpicture}
				\begin{scope}[auto, scale=1.5]
					\foreach \type/\pos/\name in {{vertex/(0,0)/a}, {vertex/(1,0)/b}, {vertex/(1.62,0.78)/c}, {vertex/(1.4,1.75)/d}, {vertex/(0.5,2.18)/e}, {vertex/(-0.4,1.75)/f}, {vertex/(-0.62,0.78)/g}, {vertex/(0.1,0.43)/t}, {vertex/(0.725,0.35)/u}, {vertex/(1.18,0.78)/v}, {vertex/(1.122,1.4)/w}, {vertex/(0.6,1.75)/x}, {vertex/(0,1.56)/y}, {vertex/(-0.22,0.97)/z}}
						\node[\type] (\name) at \pos {};
					\foreach \edgetype/\source/ \dest in {pedge/a/e, nedge/a/z, pedge/a/d, pedge/a/u, nedge/b/t, pedge/b/f, pedge/b/e, pedge/b/v, nedge/c/u, pedge/c/g, pedge/c/f, pedge/c/w, nedge/d/v, pedge/d/g, pedge/d/x, nedge/e/w, pedge/e/y, nedge/f/x, pedge/f/z, nedge/g/y, pedge/g/t}
						\path[\edgetype] (\source) -- (\dest);
				\end{scope}
			\end{tikzpicture}
	 	\caption{The sporadic maximal connected cyclotomic $\Z[\omega]$-graphs $S_{10}$, $S_{12}$, and $S_{14}$ of orders $10$, $12$, and $14$ respectively. The $\Z$-graph $S_{14}$ is also a $\Z[i]$-graph.}
		\label{fig:maxcycs5}
	\end{figure}

	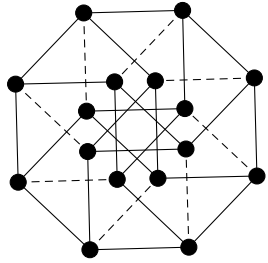
\begin{figure}[htbp]
		\centering
		\begin{tikzpicture}
			\newdimen\rad
			\rad=1.7cm
			\newdimen\radi
			\radi=0.7cm
			\foreach \x in {69,114,159,204,249,294,339,384}
			{
				\draw (\x:\radi) node[vertex] {};
				\draw[pedge] (\x:\radi) -- (\x+135:\radi);
		    }
			\foreach \x in {69,159,249,339}
			{
		    	\draw (\x:\rad) node[vertex] {};
				\draw[pedge] (\x:\rad) -- (\x+45:\rad);
				\draw[nedge] (\x:\rad) -- (\x+45:\radi);
				\draw[pedge] (\x:\rad) -- (\x-45:\radi);
		    }
			\foreach \x in {114,204,294,384}
			{
		    	\draw (\x:\rad) node[vertex] {};
				\draw[pedge] (\x:\rad) -- (\x+45:\rad);
				\draw[nedge] (\x:\rad) -- (\x+45:\radi);
				\draw[pedge] (\x:\rad) -- (\x-45:\radi);
		    }
		\end{tikzpicture}
		\caption{The sporadic maximal connected cyclotomic $\Z$-hypercube $S_{16}$.}
		\label{fig:maxcycs6}
	\end{figure}

	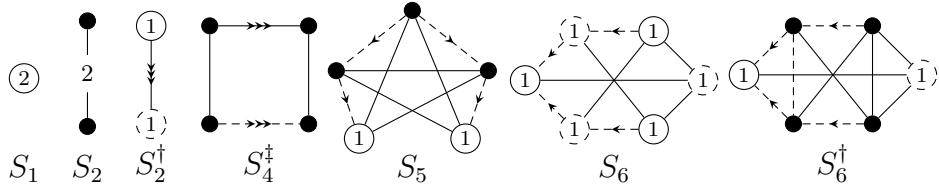
\begin{figure}[htbp]
		\centering
		\begin{tikzpicture}[scale=1.5, auto]
		\begin{scope}	
			\foreach \pos/\name/\type/\charge in {{(0,0.4)/a/pc/{2}}}
				\node[\type] (\name) at \pos {$\charge$}; 
			\node at (0,-0.4) {$S_1$};
		\end{scope}
		\end{tikzpicture}
		\begin{tikzpicture}[scale=1.4]
		\begin{scope}	
			\foreach \pos/\name/\type/\charge in {{(0,0)/a/zc/{}}, {(0,1)/b/zc/{}}}
				\node[\type] (\name) at \pos {$\charge$}; 
			\foreach \edgetype/\source/ \dest in {pedge/a/b}
			\path[\edgetype] (\source) -- node[weight2] {$2$} (\dest);
			\node at (0,-0.4) {$S_2$};
		\end{scope}
		\end{tikzpicture}
		\begin{tikzpicture}[scale=1.3, auto]
		\begin{scope}	
			\foreach \pos/\name/\type/\charge in {{(0,0)/a/nc/1}, {(0,1)/b/pc/1}}
				\node[\type] (\name) at \pos {$\charge$}; 
			\foreach \edgetype/\source/ \dest / \weight in {wwwedge/b/a/{}}
			\path[\edgetype] (\source) -- node[weight] {$\weight$} (\dest);
			\node at (0,-0.4) {$S_2^\dag$};
		\end{scope}
		\end{tikzpicture}
		\begin{tikzpicture}[scale=1.3, auto]
		\begin{scope}	
			\foreach \pos/\name/\type/\charge in {{(0,0)/a/zc/{}}, {(0,1)/b/zc/{}}, {(1,0)/c/zc/{}}, {(1,1)/d/zc/{}}}
				\node[\type] (\name) at \pos {$\charge$}; 
			\foreach \edgetype/\source/ \dest / \weight in {wwwnedge/a/c/{}, wwwedge/b/d/{}, pedge/a/b/{}, pedge/c/d/{}}
			\path[\edgetype] (\source) -- node[weight] {$\weight$} (\dest);
		\end{scope}
		\node at (0.5,-0.4) {$S_4^\ddag$};
		\end{tikzpicture}
		\begin{tikzpicture}[scale=1, auto]
		\begin{scope}	
			\foreach \pos/\name/\type/\charge in {{(-0.3,0.9)/a/zc/{}}, {(1.7,0.9)/b/zc/{}}, {(0,0)/c/pc/1}, {(1.4,0)/d/pc/1}, {(0.7,1.7)/e/zc/{}}}
				\node[\type] (\name) at \pos {$\charge$}; 
			\foreach \edgetype/\source/ \dest / \weight in {wnedge/a/c/{}, wnedge/b/d/{}, pedge/c/e/{}, pedge/e/d/{}, pedge/a/d/{}, pedge/b/c/{}, wnedge/e/b/{}, pedge/a/b/{}, wnedge/e/a/{}}
			\path[\edgetype] (\source) -- node[weight] {$\weight$} (\dest);
			\node at (0.7,-0.4) {$S_5$};
		\end{scope}
		\end{tikzpicture}
		\begin{tikzpicture}[scale=1.3, auto]
		\begin{scope}
			\foreach \pos/\name/\sign/\charge in {{(-0.1,0)/a/nc/1}, {(0.7,0)/b/pc/1}, {(1.2,0.5)/c/nc/1}, {(-0.6,0.5)/d/pc/1}, {(-0.1,1)/e/nc/1}, {(0.7,1)/f/pc/1}}
				\node[\sign] (\name) at \pos {$\charge$};
			\foreach \edgetype/\source/\dest/\weight in {{wnedge/b/a/{}}, {wnedge/a/d/{}}, {wnedge/f/e/{}}, {pedge/b/c/{}}, {wnedge/e/d/{}}, {pedge/f/c/{}}, {pedge/d/c/{}}, {pedge/a/f/{}}, {pedge/b/e/{}}}
				\path[\edgetype] (\source) -- node[weight] {$\weight$} (\dest);
				\node at (0.3,-0.4) {$S_6$};
		\end{scope}
		\end{tikzpicture}
		\begin{tikzpicture}[scale=1.3, auto]
		\begin{scope}
			\foreach \pos/\name/\sign/\charge in {{(-0.1,0)/a/zc/{}}, {(0.7,0)/b/zc/{}}, {(1.2,0.5)/c/nc/1}, {(-0.6,0.5)/d/pc/1}, {(-0.1,1)/e/zc/{}}, {(0.7,1)/f/zc/{}}}
				\node[\sign] (\name) at \pos {$\charge$};
			\foreach \edgetype/\source/\dest/\weight in {{wnedge/b/a/{}}, {wnedge/a/d/{}}, {wnedge/f/e/{}}, {pedge/b/c/{}}, {wnedge/e/d/{}}, {pedge/f/c/{}}, {pedge/d/c/{}}, {nedge/e/a/{}}, {pedge/b/f/{}}, {pedge/a/f/{}}, {pedge/b/e/{}}}
				\path[\edgetype] (\source) -- node[weight] {$\weight$} (\dest);
				\node at (0.3,-0.4) {$S_6^\dag$};
		\end{scope}
		\end{tikzpicture}
		\caption{The sporadic maximal connected cyclotomic $\Z[\omega]$-graphs of orders $1$, $2$, $4$, $5$, and $6$. The $\Z$-graphs $S_1$ and $S_2$ are also $\Z[i]$-graphs.}
		\label{fig:maxcycs8}
	\end{figure}

	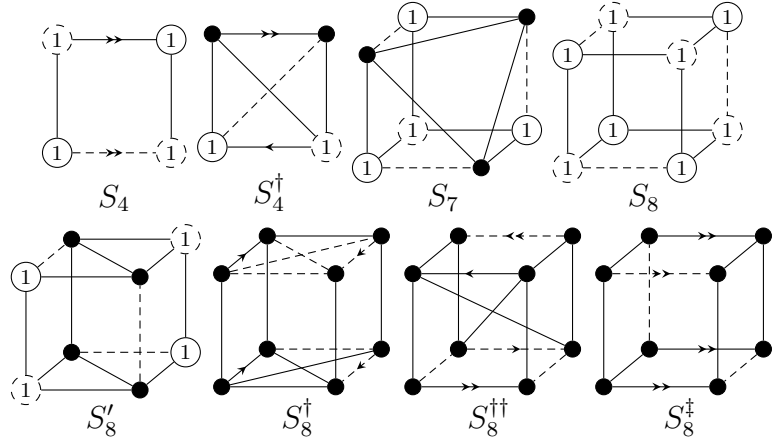
\begin{figure}[htbp]
		\centering
		\begin{tikzpicture}[scale=1.5, auto]
		\begin{scope}	
			\foreach \pos/\name/\type/\charge in {{(0,0)/a/pc/1}, {(0,1)/b/nc/1}, {(1,0)/c/nc/1}, {(1,1)/d/pc/1}}
				\node[\type] (\name) at \pos {$\charge$}; 
			\foreach \edgetype/\source/ \dest / \weight in {wwnedge/a/c/{}, wwedge/b/d/{}, pedge/b/a/{}, pedge/c/d/{}}
			\path[\edgetype] (\source) -- node[weight] {$\weight$} (\dest);
			\node at (0.5,-0.4) {$S_4$};
		\end{scope}
		\end{tikzpicture}
		\begin{tikzpicture}[scale=1.5, auto]
		\begin{scope}	
			\foreach \pos/\name/\type/\charge in {{(0,0)/a/pc/1}, {(0,1)/b/zc/{}}, {(1,0)/c/nc/1}, {(1,1)/d/zc/{}}}
				\node[\type] (\name) at \pos {$\charge$}; 
			\foreach \edgetype/\source/ \dest / \weight in {wedge/c/a/{}, wwedge/b/d/{}, pedge/b/a/{}, pedge/c/d/{}, nedge/a/d/{}, pedge/c/b/{}}
			\path[\edgetype] (\source) -- node[weight] {$\weight$} (\dest);
			\node at (0.5,-0.4) {$S_4^\dag$};
		\end{scope}
		\end{tikzpicture}
			\begin{tikzpicture}
				\def\XthreeDadj{0.6}
				\def\YthreeDadj{0.5}
				\begin{scope}
					\foreach \pos/\name/\sign/\charge in {{(0,0)/a/pc/1}, {(0,1.5)/b/zc/{}}, {(1.5,0)/c/zc/{}}, {(0 + \XthreeDadj,0 + \YthreeDadj)/d/nc/1}, {(0 + \XthreeDadj,1.5 + \YthreeDadj)/e/pc/1}, {(1.5 + \XthreeDadj,0 + \YthreeDadj)/f/pc/1}, {(1.5 + \XthreeDadj,1.5 + \YthreeDadj)/g/zc/{}}}
						\node[\sign] (\name) at \pos {$\charge$}; 
					\foreach \edgetype/\source/ \dest in {pedge/b/a, nedge/a/c, pedge/a/d, pedge/c/f, nedge/g/f, pedge/b/g, pedge/d/e, pedge/d/f, nedge/b/e, pedge/c/g, pedge/e/g, pedge/b/c}
					\path[\edgetype] (\source) -- (\dest);
				\node at (1,-0.4) {$S_{7}$};
				\end{scope}
				\end{tikzpicture}
				\begin{tikzpicture}
					\def\XthreeDadj{0.6}
					\def\YthreeDadj{0.5}
				\begin{scope}
					\foreach \pos/\name/\sign/\charge in {{(0,0)/a/nc/1}, {(0,1.5)/b/pc/1}, {(1.5,0)/c/pc/1}, {(1.5,1.5)/d/nc/1}, {(0 + \XthreeDadj,0 + \YthreeDadj)/e/pc/1}, {(0 + \XthreeDadj,1.5 + \YthreeDadj)/f/nc/1}, {(1.5 + \XthreeDadj,0 + \YthreeDadj)/g/nc/1}, {(1.5 + \XthreeDadj,1.5 + \YthreeDadj)/h/pc/1}}
						\node[\sign] (\name) at \pos {$\charge$}; 
					\foreach \edgetype/\source/ \dest in {pedge/b/a, nedge/a/c, pedge/a/e, pedge/c/g, pedge/c/d, nedge/b/f, pedge/b/d, pedge/e/f, pedge/e/g, nedge/h/g, pedge/f/h, pedge/d/h}
					\path[\edgetype] (\source) -- (\dest);
				\node at (1,-0.4) {$S_{8}$};
				\end{scope}
				\end{tikzpicture}

				\begin{tikzpicture}
					\def\XthreeDadj{0.6}
					\def\YthreeDadj{0.5}
				\begin{scope}	
					\foreach \pos/\name/\sign/\charge in {{(0,0)/a/nc/1}, {(0,1.5)/b/pc/1}, {(1.5,0)/c/zc/{}}, {(1.5,1.5)/d/zc/{}}, {(0 + \XthreeDadj,0 + \YthreeDadj)/e/zc/{}}, {(0 + \XthreeDadj,1.5 + \YthreeDadj)/f/zc/{}}, {(1.5 + \XthreeDadj,0 + \YthreeDadj)/g/pc/1}, {(1.5 + \XthreeDadj,1.5 + \YthreeDadj)/h/nc/1}}
						\node[\sign] (\name) at \pos {$\charge$}; 
					\foreach \edgetype/\source/ \dest in {pedge/b/a, pedge/a/c, pedge/a/e, pedge/c/g, nedge/c/d, nedge/b/f, pedge/b/d, pedge/e/f, nedge/e/g, pedge/h/g, pedge/f/h, pedge/d/h, pedge/d/f, pedge/c/e}
					\path[\edgetype] (\source) -- (\dest);
				\node at (1,-0.4) {$S^\prime_{8}$};
				\end{scope}
				\end{tikzpicture}
			\begin{tikzpicture}
				\def\XthreeDadj{0.6}
				\def\YthreeDadj{0.5}
			\begin{scope}	
				\foreach \pos/\name/\sign/\charge in {{(0,0)/a/zc/{}}, {(0,1.5)/b/zc/{}}, {(1.5,0)/c/zc/{}}, {(1.5,1.5)/d/zc/{}}, {(0 + \XthreeDadj,0 + \YthreeDadj)/e/zc/{}}, {(0 + \XthreeDadj,1.5 + \YthreeDadj)/f/zc/{}}, {(1.5 + \XthreeDadj,0 + \YthreeDadj)/g/zc/{}}, {(1.5 + \XthreeDadj,1.5 + \YthreeDadj)/h/zc/{}}}
					\node[\sign] (\name) at \pos {$\charge$}; 
				\foreach \edgetype/\source/ \dest/\weight in {pedge/b/a/{}, pedge/a/c/{}, pedge/a/g/{}, nedge/b/h/{}, wedge/a/e/{}, wnedge/g/c/{}, pedge/c/d/{}, wedge/b/f/{}, nedge/b/d/{}, pedge/e/f/{}, nedge/e/g/{}, pedge/h/g/{}, pedge/f/h/{}, wnedge/h/d/{}, nedge/d/f/{}, pedge/c/e/{}}
				\path[\edgetype] (\source) -- node[weight] {$\weight$} (\dest);
				\node at (1,-0.4) {$S^{\dagger}_{8}$};
			\end{scope}
			\end{tikzpicture}
			\begin{tikzpicture}
				\def\XthreeDadj{0.6}
				\def\YthreeDadj{0.5}
			\begin{scope}	
				\foreach \pos/\name/\sign/\charge in {{(0,0)/a/zc/{}}, {(0,1.5)/b/zc/{}}, {(1.5,0)/c/zc/{}}, {(1.5,1.5)/d/zc/{}}, {(0 + \XthreeDadj,0 + \YthreeDadj)/e/zc/{}}, {(0 + \XthreeDadj,1.5 + \YthreeDadj)/f/zc/{}}, {(1.5 + \XthreeDadj,0 + \YthreeDadj)/g/zc/{}}, {(1.5 + \XthreeDadj,1.5 + \YthreeDadj)/h/zc/{}}}
					\node[\sign] (\name) at \pos {$\charge$}; 
				\foreach \edgetype/\source/ \dest/\weight in {pedge/b/a/{}, wwedge/a/c/{}, nedge/a/e/{}, nedge/g/c/{}, pedge/d/c/{}, pedge/b/f/{}, wedge/d/b/{}, pedge/e/f/{}, wnedge/e/g/{}, pedge/h/g/{}, wwnedge/h/f/{}, pedge/h/d/{}, pedge/d/e/{}, pedge/b/g/{}}
				\path[\edgetype] (\source) -- node[weight] {$\weight$} (\dest);
				\node at (1,-0.4) {$S^{\dagger \dagger}_{8}$};
			\end{scope}
		\end{tikzpicture}
		\begin{tikzpicture}
			\def\XthreeDadj{0.6}
			\def\YthreeDadj{0.5}
		\begin{scope}	
			\foreach \pos/\name/\sign/\charge in {{(0,0)/a/zc/{}}, {(0,1.5)/b/zc/{}}, {(1.5,0)/c/zc/{}}, {(1.5,1.5)/d/zc/{}}, {(0 + \XthreeDadj,0 + \YthreeDadj)/e/zc/{}}, {(0 + \XthreeDadj,1.5 + \YthreeDadj)/f/zc/{}}, {(1.5 + \XthreeDadj,0 + \YthreeDadj)/g/zc/{}}, {(1.5 + \XthreeDadj,1.5 + \YthreeDadj)/h/zc/{}}}
				\node[\sign] (\name) at \pos {$\charge$}; 
			\foreach \edgetype/\source/ \dest/\weight in {pedge/b/a/{}, wwedge/a/c/{}, pedge/a/e/{}, nedge/g/c/{}, pedge/d/c/{}, pedge/b/f/{}, wwnedge/b/d/{}, nedge/e/f/{}, wwedge/e/g/{}, pedge/h/g/{}, wwedge/f/h/{}, pedge/h/d/{}}
			\path[\edgetype] (\source) -- node[weight] {$\weight$} (\dest);
			\node at (1,-0.4) {$S^{\ddag}_{8}$};
		\end{scope}
	\end{tikzpicture}

		\caption{The sporadic maximal connected cyclotomic $\Z[i]$-graphs of orders $4$, $7$, and $8$. The $\Z$-graphs $S_7$, $S_8$, and $S_8^\prime$ are also $\Z[\omega]$-graphs.}
		\label{fig:maxcycs10}
	\end{figure}
	
	We record for later use the following consequences which are immediate from the  classification.

	\begin{corollary}\label{cor:classi}
		Any connected cyclotomic $\Z[i]$-matrix that is equivalent to the adjacency matrix of a subgraph of $T_{2k}$, $T^{(i)}_{2k}$, $C_{2k}$, $C_{2k+1}$, $C_{2k}^{++}$, or $C_{2k}^{+-}$ is strongly equivalent to the adjacency matrix of a subgraph of $T_{2k}$, $T^{(i)}_{2k}$, $C_{2k}$, $\pm C_{2k+1}$, $\pm C_{2k}^{++}$, or $C_{2k}^{+-}$.
	\end{corollary}
	
	\begin{corollary}\label{cor:classw}
		Any connected cyclotomic $\Z[\omega]$-matrix that is equivalent to the adjacency matrix of a subgraph of $T_{2k}$, $T^{(\omega)}_{2k}$, $C_{2k}^{++}$, or $C_{2k}^{+-}$ is strongly equivalent to the adjacency matrix of a subgraph of $T_{2k}$, $T^{(\omega)}_{2k}$, $\pm C_{2k}^{++}$, or $C_{2k}^{+-}$.
	\end{corollary}
	
	There are two types of maximal cyclotomic graphs.
	\begin{itemize}
		\item[] The \textbf{sporadics}: $S_1$, $S_2$, $S_2^\dag$, $S_4$, $S_4^\dag$, $S_4^\ddag$, $S_5$, $S_6$, $S_6^\dag$, $S_7$, $S_8$, $S_8^\prime$, $S_8^\dag$, $S_8^{\dag \dag}$, $S_8^\ddag$, $S_{10}$, $S_{12}$, $S_{14}$, and $S_{16}$; see Figures~\ref{fig:maxcycs5}, \ref{fig:maxcycs6}, \ref{fig:maxcycs8}, and \ref{fig:maxcycs10}.
		\item[]
		\item[] The \textbf{non-sporadics}: $T_{2k} (k \geqslant 3)$, $T_{2k}^{(i)} (k \geqslant 3)$, $T_{2k}^{(\omega)} (k \geqslant 3)$, $C_{2k} (k \geqslant 2)$, $C_{2k}^{++} (k \geqslant 2)$, $C_{2k}^{+-} (k \geqslant 2)$, and $C_{2k+1} (k \geqslant 1)$; see Figures~\ref{fig:maxcycs1}, \ref{fig:maxcycs2}, \ref{fig:maxcycs3}, and \ref{fig:maxcycs4}.
	\end{itemize}
	
	A graph is called \textbf{minimal non-cyclotomic} if it has at least one eigenvalue lying outside of the interval $[-2,2]$ and is minimal in that respect, i.e., all of its subgraphs are cyclotomic.
	The two types of maximal cyclotomic graph lead to two types of minimal non-cyclotomic graph.
	A minimal non-cyclotomic graph is called \textbf{non-supersporadic} if all of its proper connected subgraphs are equivalent to subgraphs of non-sporadics and \textbf{supersporadic} otherwise.
	
	\begin{figure}[h]
		\centering
			\begin{tikzpicture}
			\begin{scope}[scale=1]	
				\foreach \pos/\name/\sign/\charge in {{(0,0)/a/pc/{\ast}}, {(0,1)/b/pc/1}, {(1,1)/c/pc/{\ast}}, {(1,0)/d/pc/{\ast}}, {(0,-0.5)/e/empty/{}}}
					\node[\sign] (\name) at \pos {$\charge$}; 
				\node at (0.5,-0.8) {$X_1$};
				\foreach \edgetype/\source/ \dest /\weight in {pedge/a/b/{\ast}, pedge/c/b/{\ast}, pedge/d/c/{\ast}}
				\path[\edgetype] (\source) -- node[weight2] {$\weight$} (\dest);
			\end{scope}
			\end{tikzpicture}
			\begin{tikzpicture}
			\begin{scope}[scale=1]	
				\foreach \pos/\name/\sign/\charge in {{(0,0)/a/pc/{\ast}}, {(0,1)/b/pc/{\ast}}, {(1,1)/c/pc/{\ast}}, {(1,0)/d/pc/{\ast}}, {(0,-0.5)/e/empty/{}}}
					\node[\sign] (\name) at \pos {$\charge$}; 
				\node at (0.5,-0.8) {$X_2$};
				\foreach \edgetype/\source/ \dest /\weight in {wwedge/c/b/{}}
				\path[\edgetype] (\source) -- node[weight] {$\weight$} (\dest);
				\foreach \edgetype/\source/ \dest /\weight in {pedge/a/b/{\ast}, pedge/d/c/{\ast}}
				\path[\edgetype] (\source) -- node[weight2] {$\weight$} (\dest);
			\end{scope}
			\end{tikzpicture}
			\begin{tikzpicture}
			\begin{scope}[scale=1]
				\foreach \pos/\name/\sign/\charge in {{(0,1)/a/pc/1}, {(0,0)/b/nc/1}}
					\node[\sign] (\name) at \pos {$\charge$}; 
				\node at (0,-0.8) {$X_3$};
				\foreach \edgetype/\source/ \dest /\weight in {pedge/a/b/{\ast}}
				\path[\edgetype] (\source) -- node[weight2] {$\weight$} (\dest);
			\end{scope}
			\end{tikzpicture}
			\begin{tikzpicture}
				\begin{scope}[scale=1]
					\foreach \pos/\name/\sign/\charge in {{(0,1)/a/pc/1}, {(0,0)/b/pc/1}}
						\node[\sign] (\name) at \pos {$\charge$}; 
					\node at (0,-0.8) {$X_4$};
					\foreach \edgetype/\source/ \dest /\weight in {wwedge/a/b/{}}
					\path[\edgetype] (\source) -- node[weight] {$\weight$} (\dest);
				\end{scope}
				\end{tikzpicture}
				\begin{tikzpicture}
				\begin{scope}[scale=1]	
					\foreach \pos/\name/\sign/\charge in {{(1,1)/a/pc/1}, {(0,1)/b/pc/1}, {(0,0)/c/pc/{\ast}}, {(0,-0.5)/d/empty/{}}}
						\node[\sign] (\name) at \pos {$\charge$}; 
					\node at (0.5,-0.8) {$X_5$};
					\foreach \edgetype/\source/ \dest /\weight in {pedge/a/b/{\ast}, pedge/c/b/{\ast}}
					\path[\edgetype] (\source) -- node[weight2] {$\weight$} (\dest);
				\end{scope}
				\end{tikzpicture}
				\begin{tikzpicture}
			\begin{scope}[scale=1, auto]	
				\foreach \pos/\name/\sign/\charge in {{(1,1)/a/pc/1}, {(0,1)/b/zc/{}}, {(0,0)/c/pc/1}}
					\node[\sign] (\name) at \pos {$\charge$}; 
				\node at (0.5,-0.8) {$X_6$};
				\foreach \edgetype/\source/ \dest /\weight in {pedge/a/b/{}, pedge/c/b/{}}
				\path[\edgetype] (\source) -- node[weight] {$\weight$} (\dest);
			\end{scope}
			\end{tikzpicture}
			\begin{tikzpicture}
			\begin{scope}[scale=1, auto]	
				\foreach \pos/\name/\sign/\charge in {{(1,1)/a/pc/1}, {(0,1)/b/zc/{}}, {(0,0)/c/nc/1}}
					\node[\sign] (\name) at \pos {$\charge$}; 
				\node at (0.5,-0.8) {$X_7$};
				\foreach \edgetype/\source/ \dest /\weight in {pedge/a/b/{}, pedge/c/b/{}}
				\path[\edgetype] (\source) -- node[weight] {$\weight$} (\dest);
			\end{scope}
			\end{tikzpicture}
			\begin{tikzpicture}
			\begin{scope}[scale=1]	
				\foreach \pos/\name/\sign/\charge in {{(1,1)/a/pc/1}, {(0,1)/b/zc/{}}, {(0,0)/c/pc/\ast}}
					\node[\sign] (\name) at \pos {$\charge$}; 
				\node at (0.5,-0.8) {$X_8$};
				\foreach \edgetype/\source/ \dest /\weight in {wwedge/a/b/{}}
				\path[\edgetype] (\source) -- node[weight] {$\weight$} (\dest);
				\foreach \edgetype/\source/ \dest /\weight in {pedge/c/b/{\ast}}
				\path[\edgetype] (\source) -- node[weight2] {$\weight$} (\dest);
			\end{scope}
			\end{tikzpicture}
		\caption{Some $\Z[i]$-graphs that are not subgraphs of any non-supersporadic graph having at least $5$ vertices.}
		\label{fig:nonnonsupersporads}
	\end{figure}
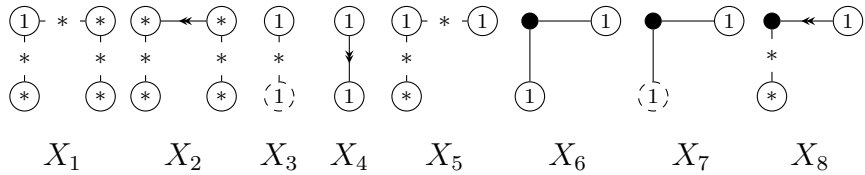
	
	Now we give some bounds on the entries of a Hermitian matrix over $\Z[i]$ or $\Z[\omega]$ having small Mahler measure.
	Let $R$ denote either $\Z[i]$ or $\Z[\omega]$.
	
	\begin{lemma}\label{lem:entryGreaterThan2}
		Let $A$ be a Hermitian $R$-matrix having an entry $a$ such that $\abs{a} > 2$.
		Then $M(R_A) \geqslant (\sqrt{5} + 1)/2 > \tau_0$.
	\end{lemma}
	\begin{proof}
		If $a$ is on the diagonal of $A$ then, by interlacing, the spectral radius of $A$ is at least $\abs{a}$.
		Otherwise, suppose that $a$ is on the off-diagonal, i.e., $a_{jk} = a$ for some $j \ne k$.
		Then the $j$th diagonal entry of $A^2$ is at least $\abs{a}^2$ and hence, by interlacing, the largest eigenvalue of $A^2$ is at least $\abs{a}^2$.
		Therefore, in either case, the spectral radius of $A$ is at least $\abs{a}$.

		Let $\rho$ be the spectral radius of $A$.
		At least one zero of $R_A$ is given by $\alpha(\rho) = (\rho + \sqrt{\rho^2 - 4})/2$.
		Up to conjugation, the smallest modulus of an element of $R$ that is greater than $2$ is at least $\sqrt{5}$.
		Hence it is clear that $\abs{\alpha(\rho)} \geqslant (\sqrt{5} + 1)/2 > \tau_0$.
	\end{proof}

	By the above lemma, in order to settle Lehmer's conjecture for polynomials $R_A$, we need only consider matrices such that each entry $a$ satisfies $\abs{a} \leqslant 2$.

	\begin{lemma}\label{lem:irratMustBeRoot2}
		Let $A$ be a non-cyclotomic $R$-matrix having at least one entry of modulus $2$.
		Then $M(R_A) \geqslant (\sqrt{5} + 1)/2 > \tau_0$.
	\end{lemma}
	\begin{proof}
		Since the Mahler measure of $R_A$ is preserved under equivalence, we may assume that $A$ has an entry $a_{jk} = 2$.

	 	\paragraph{Case 1.}
		Suppose $j = k$.
		Since the $1 \times 1$ matrix $(2)$ is cyclotomic, $A$ has at least two rows and hence contains as a principal submatrix the matrix
		\[
			\begin{pmatrix}
				2 & a \\
				a & b
			\end{pmatrix}.
		\]
		Since $A$ is indecomposable, it is possible to choose $a$ to be nonzero.
		Therefore, the $j$th diagonal entry of $A^2$ is at least $4+\abs{a}^2$, and hence the spectral radius $\rho$ of $A$ is at least $\sqrt{4+\abs{a}^2}$.
		Now, for all nonzero $a \in R$, we have $\abs{a}^2 \geqslant 1$ and so we have the following inequality
		\[
			\rho \geqslant \sqrt{4+\abs{a}^2} \geqslant \sqrt{5}.
		\]
		Hence, the associated reciprocal polynomial $R_A$ has a zero with absolute value at least $(\sqrt{5} + 1)/2$.
		Therefore $M(R_A) \geqslant (\sqrt{5} + 1)/2 > \tau_0$.

		\paragraph{Case 2.}
		Suppose $j \ne k$.
		Then $A$ contains as a principal submatrix the matrix
		\[
			\begin{pmatrix}
				a & 2 \\
				2 & b
			\end{pmatrix}.
		\]
		If either $a$ or $b$ are nonzero then by the same argument as before $M(R_A) \geqslant (\sqrt{5} + 1)/2 > \tau_0$.
		Otherwise, if both $a$ and $b$ are zero, since 
		\[
			\begin{pmatrix}
				0 & 2 \\
				2 & 0
			\end{pmatrix}
		\]
		is cyclotomic, $A$ contains as a principal submatrix the matrix
		\[
			\begin{pmatrix}
				0 & 2 & c \\
				2 & 0 & d \\
				c & d & e
			\end{pmatrix}.
		\]
		Since $A$ is indecomposable, we can choose this submatrix so that at least one of $c$ and $d$ is nonzero.
		By applying the same argument as before we obtain the inequality
		\[
			M(R_A) \geqslant (\sqrt{5} + 1)/2 > \tau_0.
		\]
	\end{proof}
	
	Given a graph $G$, a path (respectively cycle) $P$ is called \textbf{chordless} if the subgraph of $G$ induced on the vertices of $P$ is a path (respectively cycle).
	Define the \textbf{path rank} of $G$ to be the maximal number of vertices in a chordless path or cycle of $G$.
	We say that $G$ has a \textbf{profile} if its vertices can be partitioned into a sequence of $k \geqslant 3$ subsets $\mathcal C_0,\dots, \mathcal C_{k-1}$ such that either 
	\begin{itemize}
		\item two vertices $v$ and $w$ are adjacent if and only if $v \in \mathcal C_{j-1}$ and $w \in \mathcal C_{j}$ for some $j \in \{1,\dots,k-1\}$ or $v$ and $w$ are both charged vertices in the same subset
		\item[or]
		\item two vertices $v$ and $w$ are adjacent if and only if $v \in \mathcal C_{j-1}$ and $w \in \mathcal C_{j}$ for some $j \in \Z/k\Z$ or $v$ and $w$ are both charged vertices in the same subset.
	\end{itemize} 
	In the latter case, we say that the profile is \textbf{cycling}.
	Given a graph $G$ with a profile $\mathcal C = (\mathcal C_0,\dots, \mathcal C_{k-1})$ we define the \textbf{profile rank} of $G$ to be $k$, the number of subsets in the profile $\mathcal C$.

	Later we will also need the following corollaries.

	\begin{corollary}\label{cor:notcycling1}
		Let $G$ be a connected charged non-supersporadic graph.
		Then the longest chordless cycle has length $4$.
	\end{corollary}
	\begin{proof}
		Suppose $G$ contains a chordless cycle $C$ on at least $5$ vertices.
		\paragraph{Case 1.} 
		$C$ is uncharged.
		Since $G$ is charged, a charged vertex must be joined to the cycle via some path.
		Let $v$ be the intersection of the vertices of this path and the cycle.
		By deleting a vertex of the cycle that is not a neighbour of $v$ (and not $v$), we obtain a non-cyclotomic subgraph of $G$, which is impossible.
		\paragraph{Case 2.} 
		$C$ is charged.
		Since $G$ is non-supersporadic, each of its edge-weights and charges has norm at most $2$.
		It is not possible to construct a charged chordless sub-cycle on more than $4$ vertices without $X_1$ of Figure~\ref{fig:nonnonsupersporads} being a subgraph.
		Hence we are done.
	\end{proof}

	The next corollary follows with a proof similar to that of Corollary~\ref{cor:notcycling1}.

	\begin{corollary}\label{cor:notcycling2}
		Let $G$ be a connected non-supersporadic graph that has at least one edge-weight of norm $2$.
		Then the longest chordless cycle has length $4$.
	\end{corollary}

	\begin{corollary}\label{cor:chargedends}
		Let $G$ be a connected charged non-supersporadic graph having path rank $r \geqslant 5$.
		If $G$ has a profile $\mathcal C$ then $\mathcal C$ is not cycling and the charged vertices must be contained in columns at either end of $\mathcal C$. 
	\end{corollary}
	\begin{proof}
		By Corollary~\ref{cor:notcycling1}, the profile $\mathcal C$ of $G$ cannot be cycling.
		The subgraph $X_1$ of Figure~\ref{fig:nonnonsupersporads} which cannot be equivalent to a subgraph of $G$ forces the charges of $G$ to be in the first or last column of $\mathcal C$.
	\end{proof}

	Again, the next corollary has essentially the same proof.

	\begin{corollary}\label{cor:weight2ends}
		Let $G$ be a connected non-supersporadic graph that has at least one edge-weight of norm $2$ and has path rank $r \geqslant 5$.
		If $G$ has a profile $\mathcal C$ then $\mathcal C$ is not cycling and the edges of norm $2$ must be between vertices of the first two or the last two columns of $\mathcal C$.
	\end{corollary}

	In order to classify all minimal non-cyclotomic graphs we must consider all possible ways of attaching a single vertex to every cyclotomic graph.
	As it stands, we need to test an infinite number of graphs, but we will reduce the amount of work required, so that it suffices to test all the supersporadic graphs (of which there are only finitely many) and the non-supersporadic graphs on up to $10$ vertices.

	\subsection{Reduction to a finite search}

	In this section we reduce the search for minimal non-cyclotomic matrices to a finite one.
	Proposition~\ref{pro:unchargedRed} below, enables us to restrict our search for minimal non-cyclotomic graphs to a search of all non-supersporadic graphs on up to $10$ vertices and all minimal non-cyclotomic supersporadic graphs.

	\begin{lemma}
	\label{lem:profr5}
		Let $G$ be equivalent to a connected subgraph of a non-sporadic graph.
		If $G$ has path rank at least $5$ then this equals its profile rank, and its columns are uniquely determined.
		Moreover, their order is determined up to reversal or cycling.
	\end{lemma}
	\begin{proof}
	Follows from the proof of \cite[Lemma 6]{McKee:noncycISM09}.
	\end{proof}

	\begin{lemma}
	  \label{lem:profn8}
	  Let $G$ be an $n$-vertex proper connected subgraph of a non-sporadic graph where $n \geqslant 8$.
	  Then its path rank equals its profile rank and its columns are uniquely determined.
	Moreover, their order is determined up to reversal or cycling.
	\end{lemma}
	\begin{proof}
	 If $G$ has path rank at least $5$ then we can apply Lemma~\ref{lem:profr5}.
	 Since having at least $9$ vertices forces $G$ to have path rank at least $5$, we can assume that $n=8$ and that $G$ has path rank $4$.
	Let $P$ be a chordless path or cycle with maximal number of vertices $r$.
	If the maximal cycle length is equal to the maximal path length, then take $P$ to be a path.
	  Now, every \emph{proper} connected $8$-vertex subgraph of a non-sporadic graph contains a chordless path on $4$ vertices.
	  Therefore, $P$ must be a path.
	  The columns of the profile of $P$ inherited from that of $G$ are singletons.
	  Because the profile is not cycling, the column to which a new vertex can be added is completely determined by the vertices it is adjacent to in $G$.
	\end{proof}

		\begin{figure}[htbp]
			\centering
				\begin{tikzpicture}
				\begin{scope}[scale=0.8]	
					\foreach \pos/\name/\sign/\charge in {{(0,0)/a/zc/{}}, {(1,1.4)/b/pc/{\ast}}, {(2,0)/c/vertex/{}}, {(0,-0.5)/d/empty/{}}}
						\node[\sign] (\name) at \pos {$\charge$}; 
					\node at (1,-0.8) {$Y_1$};
					\foreach \edgetype/\source/ \dest /\weight in {pedge/a/b/{\ast}, pedge/c/b/{\ast}, pedge/c/a/{\ast}}
					\path[\edgetype] (\source) -- node[weight2] {$\weight$} (\dest);
				\end{scope}
				\end{tikzpicture}
				\begin{tikzpicture}
				\begin{scope}[scale=0.8]	
					\foreach \pos/\name/\sign/\charge in {{(0,0)/a/pc/{+}}, {(1,1.4)/b/zc/{}}, {(2,0)/c/pc/{+}}, {(0,-0.5)/d/empty/{}}}
						\node[\sign] (\name) at \pos {$\charge$}; 
					\node at (1,-0.8) {$Y_2$};
					\foreach \edgetype/\source/ \dest /\weight in {pedge/a/b/{}, pedge/c/b/{}, pedge/c/a/{}}
					\path[\edgetype] (\source) -- node[weight] {$\weight$} (\dest);
				\end{scope}
				\end{tikzpicture}
				\begin{tikzpicture}[scale=1]
					\begin{scope}[xshift=5.5cm, yshift=-0.5cm]	
						\foreach \pos/\name in {{(0,0)/a}, {(0,1)/b}, {(1,1)/c}, {(1,0)/d}}
							\node[vertex] (\name) at \pos {}; 
						\node at (0.5,-0.8) {$Y_3$};
						\foreach \edgetype/\source/ \dest /\weight in {pedge/a/b/{}, pedge/b/c/{}, wwedge/c/d/{}, wwedge/a/d/{}}
						\path[\edgetype] (\source) -- (\dest);
					\end{scope}
					\end{tikzpicture}
				\begin{tikzpicture}[scale=1]
					\begin{scope}[xshift=5.5cm, yshift=-0.5cm]	
						\foreach \pos/\name in {{(0,0)/a}, {(0,1)/b}, {(1,1)/c}, {(1,0)/d}}
							\node[vertex] (\name) at \pos {}; 
						\node at (0.5,-0.8) {$Y_4$};
						\foreach \edgetype/\source/ \dest /\weight in {pedge/a/b/{}, pedge/b/c/{}, pedge/c/d/{}, wedge/a/d/{}}
						\path[\edgetype] (\source) -- node[weight] {$\weight$} (\dest);
					\end{scope}
					\end{tikzpicture}
					\begin{tikzpicture}
					\begin{scope}[yshift=-0.5cm,scale=1]	
						\foreach \pos/\name in {{(0,0)/a}, {(0,1)/b}, {(1,1)/c}, {(1,0)/d},{(2,0)/f}}
							\node[vertex] (\name) at \pos {}; 
						\node at (1,-0.8) {$Y_5$};
						\foreach \edgetype/\source/ \dest /\weight in {pedge/a/b/{}, pedge/b/c/{}, pedge/c/d/{}, pedge/a/d/{}, pedge/d/f/{}}
						\path[\edgetype] (\source) -- node[weight] {$\weight$} (\dest);
					\end{scope}
					\end{tikzpicture}
					\begin{tikzpicture}[scale=1]
				\begin{scope}[xshift=2.75cm, yshift=-0.5cm]	
					\foreach \pos/\name in {{(0,0)/a}, {(0,1)/b}, {(1,1)/c}, {(1,0)/d}, {(2,1)/e}, {(2,0)/f}}
						\node[vertex] (\name) at \pos {}; 
					\node at (1,-0.8) {$Y_6$};
					\foreach \edgetype/\source/ \dest /\weight in {pedge/a/b/{}, pedge/c/d/{}, pedge/a/d/{}, pedge/d/f/{}, pedge/e/f/{}}
					\path[\edgetype] (\source) -- node[weight] {$\weight$} (\dest);
				\end{scope}
				\end{tikzpicture}

				\begin{tikzpicture}
				\begin{scope}[scale=1]
					\newdimen\rad
					\rad=0.6cm

					\foreach \x in {90,162,234,306,378}
					{
				    	\draw (\x:\rad) node[vertex] {};
						\draw[pedge] (\x:\rad) -- (\x+72:\rad);
				    }
					\draw (338:1.3cm) node[vertex] {};
				\draw[pedge] (306:\rad) -- (338:1.3cm);
				\end{scope}
				\begin{scope}[xshift=-0.2cm, yshift=-0.5cm]
					\node at (0.5,-0.8) {$Y_7$};
				\end{scope}
			\end{tikzpicture}
			\begin{tikzpicture}
				\begin{scope}[scale=1]
					\newdimen\rad
					\rad=0.6cm

					\foreach \x in {90,162,234,306,378}
					{
				    	\draw (\x:\rad) node[vertex] {};
						\draw[pedge] (\x:\rad) -- (\x+72:\rad);
				    }
				    \draw[wedge] (234:\rad) -- (306:\rad);
					\draw (338:1.3cm) node[vertex] {};
				\draw[pedge] (306:\rad) -- (338:1.3cm);
				\end{scope}
				\begin{scope}[xshift=-0.2cm, yshift=-0.5cm]
					\node at (0.5,-0.8) {$Y_8$};
				\end{scope}
			\end{tikzpicture}
			\begin{tikzpicture}
				\begin{scope}[scale=1, auto]
					\foreach \pos/\name/\sign/\charge in {{(0,0)/a1/pc/+}, {(1,1)/b1/pc/\ast}, {(2,1)/c1/pc/\ast}, {(2,0)/d1/pc/\ast}, {(0,1)/c2/pc/\ast}}
						\node[\sign] (\name) at \pos {$\charge$};
					\foreach \edgetype/\source/ \dest in {pedge/a1/b1, pedge/b1/c1, pedge/d1/c1, pedge/b1/c2}
						\path[\edgetype] (\source) -- (\dest);
					\node at (1,-0.8) {$Y_9$};
				\end{scope}
			\end{tikzpicture}
			\begin{tikzpicture}
				\begin{scope}[scale=1, auto]
					\foreach \pos/\name/\sign/\charge in {{(0,0)/a1/pc/\ast}, {(1,1)/b1/pc/\ast}, {(2,1)/c1/pc/\ast}, {(2,0)/d1/pc/\ast}, {(0,1)/c2/pc/\ast}}
						\node[\sign] (\name) at \pos {$\charge$};
					\foreach \edgetype/\source/ \dest in {wwedge/a1/b1, pedge/b1/c1, pedge/d1/c1, pedge/b1/c2}
						\path[\edgetype] (\source) -- (\dest);
					\node at (1,-0.8) {$Y_{10}$};
				\end{scope}
			\end{tikzpicture}
			\caption{Some $\Z[i]$-graphs that are \emph{not} subgraphs of any non-supersporadic graph on at least $10$ vertices.}
			\label{fig:Ygraphs}
		\end{figure}
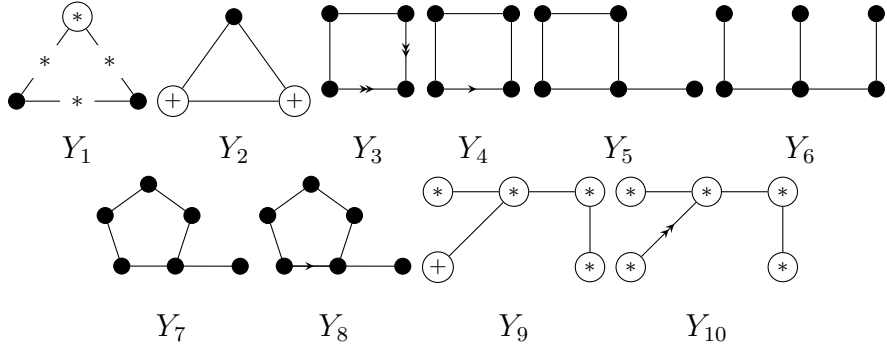

	The graphs in Figure~\ref{fig:Ygraphs} are \emph{not} subgraphs of any non-supersporadic graph on at least $10$ vertices.
	In the following proposition we shall treat only the ring $\Z[i]$ since the arguments are the same, if not slightly simpler, for $\Z[\omega]$.


	\begin{proposition}
		\label{pro:unchargedRed}
	 Let $G$ be a connected non-supersporadic $\Z[i]$-graph with $n \geqslant 10$ vertices.
	Then $G$ is equivalent to a subgraph of a non-sporadic $\Z[i]$-graph.
	\end{proposition}

	\begin{proof}
		Let $G$ satisfy the assumption of the proposition.
		Take a chordless path or cycle $P$ of $G$ with maximal number of vertices.
		Let $x$ and $y$ be the endvertices of $P$ if $P$ is a path, otherwise if $P$ is a cycle, let $x$ and $y$ be any two adjacent vertices of $P$.
		If there simultaneously exist chordless paths and chordless cycles in $G$ both containing the maximal number of vertices then we take $P$ to be one of the paths. 
		\begin{claim}
			\label{clm:connected}
			The subgraphs $G \backslash \left \{x\right \}$, $G \backslash \left \{y\right \}$, and $G \backslash \left \{x,y\right \}$ are connected.
		\end{claim}
		If a vertex $x^\prime$ of $G$ is adjacent to $x$, then it must also be adjacent to another vertex of $P$ otherwise, if $P$ is a path, there exists a longer chordless path or, if $P$ is a cycle, there exists a chordless path with the same number of vertices of $P$.
		It follows that $G \backslash \left \{x\right \}$ is connected.
		This is similar for $G \backslash \left \{y\right \}$.

		Now suppose $P$ is a chordless cycle and there exists some vertex $z$ not on $P$ that is adjacent to both $x$ and $y$ and to no other vertex of $P$.
		Since the graph $Y_1$ cannot be equivalent to any subgraph of $G$, the triangle $xyz$ must have at least two charges, and hence at least one of $x$ or $y$ is charged.
		Therefore $P$ is a charged chordless cycle of length at least $5$.
		But, by Corollary~\ref{cor:notcycling1}, the longest charged chordless cycle of $G$ has length $4$, which gives a contradiction.
		Therefore, $G \backslash \left \{x,y\right \}$ is connected and we have proved Claim~\ref{clm:connected}.

		Next we show that $G$ has a profile.
		\begin{claim}
			\label{clm:profile}
			$G$ has a profile.
		\end{claim}
		Since $G$ is non-sporadic, the connected subgraphs $G \backslash \left \{x\right \}$ and $G \backslash \left \{y\right \}$ are cyclotomic.
		Moreover, since they have at least $9$ vertices, they must have path rank at least $5$, and hence, by Lemma~\ref{lem:profr5}, they have uniquely determined profiles.
		The connected subgraph $G\backslash \left \{x,y\right \}$ has at least $8$ vertices and is a proper subgraph of the cyclotomic graph $G \backslash \left \{x\right \}$.
		Thus, by Lemma~\ref{lem:profn8}, $G\backslash \left \{x,y\right \}$ has a uniquely determined profile.

		By Corollary~\ref{cor:classi}, we can switch $G \backslash \left \{x\right \}$ to obtain a subgraph $G_x$ of one of $T_{2k}$, $T^{(i)}_{2k}$, $C_{2k}$, $\pm C_{2k+1}$, $\pm C_{2k}^{++}$, and $C_{2k}^{+-}$ for some $k$.
		We can simultaneously switch $G \backslash \left \{y\right \}$ to obtain a subgraph $G_y$ so that $G_x \backslash \left \{y\right \}$ and $G_y \backslash \left \{x\right \}$ are the same subgraph which we will call $G_{xy}$.

		Let $\mathcal C^\prime$ be the profile of $G_{xy}$.
		Since it is uniquely determined, the profile of $G_x$ (respectively $G_y$) can be obtained by the addition of $x$ (respectively $y$) to the profile $\mathcal C^\prime$ of $G_{xy}$.
		In particular, to obtain the profile of $G_x$ (respectively $G_y$), either the vertex $x$ (respectively $y$) is given its own column at one end of $\mathcal C^\prime$ or it is added to the first or last column of $\mathcal C^\prime$.
		Now merge the profiles of $G_x$ and $G_y$ to give a sequence of columns $\mathcal C$.
		We will show that $\mathcal C$ is a profile for $G$.

		Let $q \geqslant 5$ be the number of columns of $\mathcal C$ and denote by $\mathcal C_j$ the $j$th column of $\mathcal C$, where $\mathcal C_{q-1}$ is the column containing $x$ and $\mathcal C_0$ is the column containing $y$.
		Suppose that no vertex of column $\mathcal C_0$ is adjacent to any vertex of column $\mathcal C_{q-1}$.
		Since no subgraph of $G$ can be equivalent to $X_1$, $X_2$, $X_3$, $X_4$, $X_5$, $Y_2$, $Y_3$, $Y_9$, or $Y_{10}$, we have that $G$ is equivalent to a subgraph of a non-sporadic graph, and moreover $\mathcal C$ is a profile for $G$. 

		Otherwise, suppose that some vertex $u$ of $\mathcal C_{q-1}$ is adjacent to some vertex $v$ of $\mathcal C_0$ with $w(u,v) = s$ for some $s \in \Z[i]$.
		By Corollary~\ref{cor:notcycling1}, since $G$ has a chordless cycle of length at least $5$ it must be uncharged and moreover, by Corollary~\ref{cor:notcycling2} each edge-weight of $G$ must be in the group of units of $\Z[i]$.
		Furthermore, any other vertex in $\mathcal C_{q-1}$ must be adjacent to some vertex in $\mathcal C_0$ otherwise it would contain a subgraph equivalent to $Y_6$, $Y_7$, or $Y_8$, which is impossible.
		Hence $\mathcal C$ is a profile for $G$ as required for Claim~\ref{clm:profile}.

		It remains to demonstrate that, when $\mathcal C$ is cycling, $G$ is equivalent to a non-sporadic graph.
		In particular, we will show that $G$ is equivalent to a subgraph of $T_{2k}$, $T^{(i)}_{2k}$, or $C_{2k}$.
		By above, we have that $G$ is uncharged and, since no subgraph of $G$ can be equivalent to $Y_1$, we have that $G$ is triangle-free.
		And for $j \in \left \{ 1,\dots,q-1 \right \}$ we have, for any vertex $a$ in $\mathcal C_j$ and vertices $b$ and $b^\prime$ in $\mathcal C_{j+1}$, the equality $w(a,b) = w(a,b^\prime)$, and for any vertex $b$ in $\mathcal C_{j+1}$ and vertices $a$ and $a^\prime$ in $\mathcal C_j$, the equality $w(a,b) = -w(a^\prime,b)$.
		Each of these edge-weights is $\pm 1$.

		Suppose there exists another vertex $u^\prime$ in $\mathcal C_{q-1}$ that is adjacent to $v$.
		Let $z$ be a vertex in the column $\mathcal C_{q-2}$, which has $w(z,u) = w(z,u^\prime)$.
		Thus, in order for $G$ to avoid containing a subgraph equivalent to either $Y_4$ or $Y_5$, we must have that $w(u^\prime,v) = -s$.
	    Similarly, if there exists another vertex $v^\prime$ in $\mathcal C_0$ which is adjacent to $u$, then $w(u,v^\prime) = s$.
		And hence, if there exist vertices $u^\prime$ and $v^\prime$ different from $u$ and $v$ with $u^\prime$ in $\mathcal C_{q-1}$ and $v^\prime$ in $\mathcal C_0$, then  $w(u^\prime,v) = -s$ and $w(u,v^\prime) = s$.
		Since $z$ in the column $\mathcal C_{q-2}$ has $w(z,u) = w(z,u^\prime)$, in order for $G$ to avoid containing a subgraph equivalent to either $Y_4$ or $Y_5$, we must have that $w(u^\prime,v^\prime) = -s$.
		Therefore, $G$ is equivalent to a subgraph of $T_{2k}$, $T^{(i)}_{2k}$, or $C_{2k}$ for some $k$.	
	\end{proof}

	\section{Computations and finite search}
	
	In this section we give details of the finite search.
	Let $R$ denote either the Gaussian integers or the Eisenstein integers.
	For $k \in \N$, define the sets
	\[
		\mathcal L_k = \{ x \in R : x\bar x = k \}.
	\]
	Let $G$ be an $R$-graph.
	Define the \textbf{degree} $d_v$ of a vertex $v$ as
	\[
		d_v = \sum_{u\in V(G)}\abs{w(u,v)}^2. 
	\]
	
	\subsection{Hermitian $R$-matrices with large norm entries}
	\begin{proposition} 
		Let $A$ be a Hermitian $R$-matrix and define $w_{max}=\displaystyle\max_{u\neq v}{A_{uv}A_{vu}}$ and $x_{max}=\max{|A_{uu}|}$. If $w_{max}\geqslant3$ or $x_{max}\geqslant2$, then $M(A)\geqslant 1.556\ldots$.
	\end{proposition}
	\begin{proof} By Section 6.2 of \cite{GTay:thesis10} we have $x_{max}\geqslant3\Rightarrow M(A)>2.618\ldots$; else if $w_{max}\geqslant5$ then $M(A)>2.369\ldots$, or if $x_{max}=2$ then $M(A)>1.722\ldots$. Otherwise either $w_{max}=4$ and $M(A)>2.081\ldots$ or $w_{max}=3$ and $M(A)>1.556\ldots$.
	\end{proof}

	Thus if $A$ is a minimal non-cyclotomic $R$-matrix with Mahler measure less than $1.3$, then $A$ has diagonal entries of absolute value at most $1$ and off-diagonal entries of norm at most $2$; thus $A$ can be represented by a (charged) $R$-graph with edges of weight at most $2$ and vertices of charge $0,1$ or $-1$ only.
	
	\subsection{Growing Algorithms}

	Given a maximal cyclotomic graph $G$ and an induced subgraph $G^\prime$, we may recover $G$ from $G^\prime$ by reintroducing the `missing' vertices one at a time, giving a sequence of cyclotomic supergraphs of $G^\prime$ contained in $G$. We may thus recover \emph{all} cyclotomic supergraphs of $G^\prime$ by considering all possible additions of a new vertex to $G$. If an addition yields a connected cyclotomic graph we describe it as a \emph{cyclotomic addition}, otherwise as a \emph{non-cyclotomic addition}: maximal graphs are therefore those which admit no cyclotomic addition.

	For an $n\times n$ matrix representative $A$ of $G$, the addition of an extra vertex is specified by a nonzero column vector $\mathbf{c} \in R^n$ and a charge $x$ from charge set $X=\{0,1,-1\}$, giving a supermatrix
\[\begin{pmatrix}
A_{11} & \cdots & A_{1n} & c_1\\
\vdots & \ddots &\vdots & \vdots\\
A_{n1} & \cdots & A_{nn} & c_n\\
\overline{c_1} & \cdots & \overline{c_n} & x
\end{pmatrix}\]

	Define $C_n(R)$ to be the collection of nonzero vectors from $R^n$.
	We define an equivalence relation on column vectors whereby $\mathbf{c}\sim \mathbf{c}^\prime$ if $c_j=\mu c^\prime_j$ for all $j$, where $\mu \in R$ is a unit. For $\mathbf{c}\sim \mathbf{c}^\prime$, the supermatrix formed from $A$, $\mathbf{c}$, and $x$ is equivalent to the one formed from $A$, $\mathbf{c}^\prime$, and $x$ by $\mu$-switching at the extra vertex, so we restrict our attention to reduced column sets $\mathfrak C_n(R)=C_n(R)/\sim$. Further, we may sometimes restrict our attention to bounded column sets $\mathfrak C_n^b(R)=\{ \mathbf{c} \in \mathfrak C_n(R)\,|\, \sum_1^n \mbox{Norm}(c_i) \leqslant b\}$.
	
	Starting with the set $\Sigma_2$ of $2$-vertex cyclotomic graphs we may therefore iteratively recover supersets $\Sigma_j$ and $T_j$ of cyclotomic and minimal non-cyclotomic supergraphs by rounds of growing followed by reduction modulo equivalence of the $\Sigma_j$. In practice such reduction becomes computationally infeasible very rapidly, so to combat combinatorial explosion we seek to fix features (such as specific subgraphs), grow representatives of all graphs exhibiting those features, then exclude those features in subsequent growings to prevent duplication. 

	\subsection{Graphs with at most ten vertices}

	\subsubsection{$\Z[i]$-graphs with a weight-$2$ edge and at most ten vertices}

	\begin{proposition}\label{nochargeandw2} Let $G$ be a minimal non-cyclotomic $\Z[i]$-graph with a weight-$2$ edge incident at a charged vertex. Then $G$ has Mahler measure at least $1.401\dots$.
	\end{proposition}

	\begin{proof} Up to equivalence, $G$ induces as subgraph one of the graphs

	\[H_1:=
	\begin{tikzpicture}[scale=1.3, auto]
		\begin{scope}	
			\foreach \pos/\name/\type/\charge in {{(0,0)/a/pc/1}, {(1,0)/b/pc/1}}
				\node[\type] (\name) at \pos {$\charge$}; 
			\foreach \edgetype/\source/ \dest / \weight in {wwedge/a/b/{}}
			\path[\edgetype] (\source) -- node[weight] {$\weight$} (\dest);
		\end{scope}
		\end{tikzpicture}\hspace{2em}
	H_2:=
	\begin{tikzpicture}[scale=1.3, auto]
		\begin{scope}	
			\foreach \pos/\name/\type/\charge in {{(0,0)/a/pc/1}, {(1,0)/b/nc/1}}
				\node[\type] (\name) at \pos {$\charge$}; 
			\foreach \edgetype/\source/ \dest / \weight in {wwedge/a/b/{}}
			\path[\edgetype] (\source) -- node[weight] {$\weight$} (\dest);
		\end{scope}
		\end{tikzpicture}\hspace{2em}
	H_3:=
	\begin{tikzpicture}[scale=1.3, auto]
		\begin{scope}	
			\foreach \pos/\name/\type/\charge in {{(0,0)/a/pc/1}, {(1,0)/b/zc/{}}}
				\node[\type] (\name) at \pos {$\charge$}; 
			\foreach \edgetype/\source/ \dest / \weight in {wwedge/a/b/{}}
			\path[\edgetype] (\source) -- node[weight] {$\weight$} (\dest);
		\end{scope}
		\end{tikzpicture}\hspace{2em}
	\]

	The graph $H_1$ is non-cyclotomic (with Mahler measure $1.883\ldots$) and therefore minimal. Growing the seed graphs $H_2,H_3$ with column sets $\mathfrak C_n(\mathcal{L}_1\cup\mathcal{L}_2\cup\{0\})$, charge set $X=\{0,1,-1\}$ terminates at $n=4$ since they have neither cyclotomic nor minimally non-cyclotomic $5$-vertex supergraphs. The $3$- and $4$-vertex minimal non-cyclotomic supergraphs recovered during this process have Mahler measure at least $1.401\ldots$ and $1.847\ldots$ respectively. 
 \end{proof}
 
 	Thus if $uv$ is an edge of weight-$2$ in a minimal non-cyclotomic graph $G$, we may assume that vertices $u$ and $v$ are uncharged. If $G$ is assumed to have an edge of weight-$2$, then (up to equivalence) it can be grown from the seed graph 
 
	\[H_4:=
	\begin{tikzpicture}[scale=1.3, auto]
		\begin{scope}	
			\foreach \pos/\name/\type/\charge in {{(0,0)/a/zc/{}}, {(1,0)/b/zc/{}}}
				\node[\type] (\name) at \pos {$\charge$}; 
			\foreach \edgetype/\source/ \dest / \weight in {wwedge/a/b/{}}
			\path[\edgetype] (\source) -- node[weight] {$\weight$} (\dest);
		\end{scope}
		\end{tikzpicture};
	\]

	Further, in each growing round we may restrict to columns and charges $\mathbf{c}$ and $x$ such that $x=0$ and $\mathbf{c} \in \mathfrak C_n(\mathcal{L}_1\cup\mathcal{L}_2\cup\{0\})$ or $x=\pm1$ and $\mathbf{c}\in \mathfrak C_n(\mathcal{L}_1\cup\{0\})$.

	From $\Sigma_2=\{H_4\}$ we generate the sets $\Sigma_3,\Sigma_4,\Sigma_5,\Sigma_6$ and $T_3,T_4,T_5,T_6$ in this way, reducing by $U_n(\Z[i])$ for $n=3,4$ and partially reducing by $U_n(\Z)$ for $n=5,6$. 

	For subsequent rounds, we note the following:

	\begin{proposition}\label{wdegmax4} Let $G$ be a minimal non-cyclotomic $R$-graph with all edges of weight at most $2$ and at least $7$ vertices. Then each vertex of $G$ has degree at most $4$.
	\end{proposition}
	\begin{proof} Let $v$ be a vertex of $G$ with degree greater than $4$. Then we require at most $5$ neighbours of $v$ such that the induced subgraph $G^\prime$ on $v$ and those neighbours also contains a vertex of degree greater than $4$. As $G$ has at least $7$ vertices, $G^\prime$ is a proper subgraph of a minimally non-cyclotomic graph and thus is cyclotomic. But by Lemma~\ref{lem:cycclass} this is impossible.\end{proof}

	For $n\geqslant6$ we may therefore restrict to columns and charges $\mathbf{c}$ and $x$ such that $x=0$ and $\mathbf{c}\in \mathfrak C_n^4(\mathcal{L}_1\cup\mathcal{L}_2\cup\{0\})$ or $x=\pm1$ and $\mathbf{c}\in \mathfrak C_n^3(\mathcal{L}_1\cup\{0\})$ to ensure the vertex being added has degree at most $4$. Further, for each graph to be grown we may discard from the column set any $c$ which would induce a vertex of degree greater than $4$ in the supergraph. In this way, we generate the sets $\Sigma_7,\Sigma_8,\Sigma_9,\Sigma_{10}$ and $T_7,T_8,T_9,T_{10}$. Brute-force reduction of the $\Sigma_j$ is not possible, but the resulting sets $T_j$ are small enough to be reduced by hand. 

	This search takes $25$ cpudays to complete up to ten vertices; Table~\ref{searchresults1} summarises the results.

	\begin{table}[H]
	\begin{center}
	\begin{tabular}{|c|c|c|l|}
	\hline
	$j$ & $|T_j|$ & $\displaystyle\min_{A\in T_j}{M(A)}$ \\\hline
	$3$ & $5$ & $1.582\ldots$\\
	$4$ & $50$ & $1.401\ldots$\\
	$5$ & $\leqslant 23$ & $1.351\ldots$\\
	$6$ & $\leqslant 39$ & $1.401\ldots$\\
	$7$ & $1$ & $1.506\ldots$\\
	$8$ & $1$ & $1.457\ldots$\\
	$9$ & $1$ & $1.425\ldots$\\
	$10$ & $1$ & $1.401\ldots$\\\hline
	\end{tabular}
	\end{center}
	\caption{Least Mahler measures of small $\Z[i]$-graphs with at least one weight-$2$ edge.}
	\label{searchresults1}
	\end{table}

	\begin{corollary} Let $G$ be a minimal noncyclotomic $\Z[i]$-graph with at most ten vertices and a weight-$2$ edge. Then $G$ has Mahler measure greater than $\tau_0$.
	\end{corollary}
	
	\begin{proof} If $G$ has a weight-$2$ edge incident at a charged vertex, then it has Mahler measure at least $1.401\ldots$ by Proposition \ref{nochargeandw2}. Otherwise $G$ can be grown from  a graph equivalent to $H_4$, so there is a representative of $G$ in one of the $T_j$ computed above. But then $G$ has Mahler measure at least $1.351$.\end{proof}

	We may thus restrict to graphs with all edges of weight $1$.
	
	\subsubsection{Graphs containing triangles}

	\begin{proposition} Let $G$ be a minimal noncyclotomic graph inducing as subgraph a triangle with all vertices charged. Then $G$ has Mahler measure greater than $\tau_0$.
	\end{proposition}

	\begin{proof} Let $H$ be such a triangle. Then $H$ is either noncyclotomic (and thus all of $G$ with Mahler measure at least $1.556\ldots$), or equivalent one of the following $\Z$-graphs
	
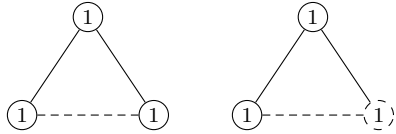
\[
	\begin{tikzpicture}
			\begin{scope}[scale=1]	
				\foreach \pos/\name/\sign/\charge in {{(90:0.8)/a/pc/1}, {(210:1)/b/pc/1}, {(330:1)/c/pc/1}}
					\node[\sign] (\name) at \pos {$\charge$}; 
				\foreach \edgetype/\source/ \dest /\weight in {pedge/a/b/{},pedge/a/c/{},nedge/b/c/{}}
				\path[\edgetype] (\source) -- (\dest);
			\end{scope}
			\end{tikzpicture}
	\hspace{2em}\begin{tikzpicture}
			\begin{scope}[scale=1]	
				\foreach \pos/\name/\sign/\charge in {{(90:0.8)/a/pc/1}, {(210:1)/b/pc/1}, {(330:1)/c/nc/1}}
					\node[\sign] (\name) at \pos {$\charge$}; 
				\foreach \edgetype/\source/ \dest /\weight in {pedge/a/b/{},pedge/a/c/{},nedge/b/c/{}}
				\path[\edgetype] (\source) -- (\dest);
			\end{scope}
			\end{tikzpicture}		
	\]
	For both rings, growing terminates after two rounds, as there are no $5$-vertex cyclotomic graphs inducing a triple-charged triangle as subgraph. Classes of $4$- and $5$-vertex minimal non-cyclotomic graphs are obtained; these have Mahler measure at least $1.58\ldots$ and $2.618\ldots$ respectively.
	\end{proof}

	\begin{proposition}  Let $G$ be a minimal non-cyclotomic graph inducing as subgraph a triangle with a single charged vertex. Then $G$ has Mahler measure greater than $\tau_0$.
	\end{proposition}
	\begin{proof} Such a triangle is either: minimal non-cyclotomic with Mahler measure at least $1.506\ldots$; equivalent to the $\Z$-graph $H_1$ given below; or ($\Z[\omega]$ only) equivalent to the graph $H_2$ given below.
	
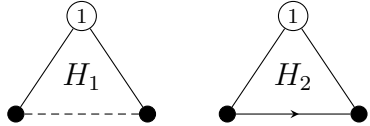
\[
	\begin{tikzpicture}
			\begin{scope}[scale=1]	
				\foreach \pos/\name/\sign/\charge in {{(90:0.8)/a/pc/1}, {(210:1)/b/zc/{}}, {(330:1)/c/zc/{}}}
					\node[\sign] (\name) at \pos {$\charge$}; 
				\foreach \edgetype/\source/ \dest /\weight in {pedge/a/b/{},pedge/a/c/{},nedge/b/c/{}}
				\path[\edgetype] (\source) -- (\dest);
				\node at (0,0) {$H_1$};
			\end{scope}
			\end{tikzpicture}
	\hspace{2em}\begin{tikzpicture}
			\begin{scope}[scale=1]	
				\foreach \pos/\name/\sign/\charge in {{(90:0.8)/a/pc/1}, {(210:1)/b/zc/{}}, {(330:1)/c/zc/{}}}
					\node[\sign] (\name) at \pos {$\charge$}; 
				\foreach \edgetype/\source/ \dest /\weight in {pedge/a/b/{},pedge/a/c/{},wedge/b/c/{}}
				\path[\edgetype] (\source) -- (\dest);
				\node at (0,0) {$H_2$};
			\end{scope}
			\end{tikzpicture}			
	\]
	Growing (subject to Proposition \ref{wdegmax4} in later rounds), we recover classes of $4$-vertex minimal non-cyclotomic graphs with Mahler measure at least $1.401\ldots$; $5$-vertex minimal non-cyclotomic graphs such that the only examples with small Mahler measure are $\Z$-graphs; and for each ring a single 6-vertex minimal non-cyclotomic graph with Mahler measure $1.425\ldots$. No $7$- or $8$-vertex minimal non-cyclotomic graphs are found, and since (by Lemma \ref{lem:cycclass}) there are no $8$-vertex cyclotomic graphs with a single-charged triangle as subgraph, this completes the search.
	\end{proof}

	\begin{proposition}  Let $G$ be a minimal non-cyclotomic graph inducing as subgraph an uncharged triangle. Then $G$ has Mahler measure greater than $\tau_0$.
	\end{proposition}

	\begin{proof} The possible triangles are all cyclotomic; for each ring, there are two distinct equivalence classes. Growing (subject to Proposition \ref{wdegmax4} in later rounds) terminates at 7 vertices for $\Z[\omega]$ and 8 for $\Z[i]$, since in the latter case there is a class of 8-vertex cyclotomic graphs inducing an uncharged triangle. 
	For $\Z[\omega]$, the only minimal non-cyclotomic classes have 4 or 5 vertices, with Mahler measure at least $1.2806\ldots$ or $1.635\ldots$ respectively; one of the $4$-vertex graphs is a new example of a graph with small Mahler measure, although there exist $\Z$-graphs with the same associated polynomial and hence Mahler measure.
	For $\Z[i]$, there are classes of minimal non-cyclotomic graphs with 4, 5 and 6 vertices; none have small Mahler measure, with lower bounds of $1.401\ldots$, $1.351\ldots$ and $1.401\ldots$ respectively.
\end{proof}

	The remaining case is graphs such that any triangular subgraph contains precisely two charges, with at least one such subgraph. Due to the existence of the family $C_{2k}^{++}$ of cyclotomic graphs, there are potentially infinitely many minimal non-cyclotomic graphs satisfying these conditions. We therefore restrict our attention to graphs with at most ten vertices.

	Let $H$ be a charged triangular subgraph of a minimal noncyclotomic graph $G$. Then (up to equivalence) $H$ is either noncyclotomic with Mahler measure at least $1.401\ldots$; equivalent to the $\Z$-graph $H_1$ below; or ($\Z[i]$ only) equivalent to the graph $H_2$ below.
	
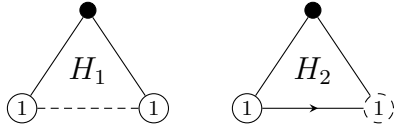
\[
	\begin{tikzpicture}
			\begin{scope}[scale=1]	
				\foreach \pos/\name/\sign/\charge in {{(90:0.8)/a/zc/{}}, {(210:1)/b/pc/1}, {(330:1)/c/pc/1}}
					\node[\sign] (\name) at \pos {$\charge$}; 
				\foreach \edgetype/\source/ \dest /\weight in {pedge/a/b/{},nedge/b/c/{},pedge/a/c/{}}
				\path[\edgetype] (\source) -- (\dest);
				\node at (0,0) {$H_1$};
			\end{scope}
			\end{tikzpicture}
	\hspace{2em}\begin{tikzpicture}
			\begin{scope}[scale=1]	
				\foreach \pos/\name/\sign/\charge in {{(90:0.8)/a/zc/{}}, {(210:1)/b/pc/1}, {(330:1)/c/nc/1}}
					\node[\sign] (\name) at \pos {$\charge$}; 
				\foreach \edgetype/\source/ \dest /\weight in {pedge/a/b/{},wedge/b/c/{},pedge/c/a/{}}
				\path[\edgetype] (\source) -- (\dest);
				\node at (0,0) {$H_2$};
			\end{scope}
			\end{tikzpicture}			
	\]

	We grow the sets $\Sigma_j$ and $T_j$ for $4\leqslant j \leqslant 10$, filtering the column sets for each matrix to eliminate supergraphs containing single-, triple- or uncharged triangles (which are covered by the earlier propositions) and, for sufficiently large graphs, vertices of degree greater than $4$ (by Proposition \ref{wdegmax4}). 

	The only nonempty $T_j$ are $T_5$ and ($\Z[i]$ only) $T_4$, with all minimal non-cyclotomic classes having Mahler measure at least $1.458\ldots$. Thus we conclude:

	\begin{corollary} Let $G$ be a minimal non-cyclotomic graph of at most ten vertices, inducing a charged triangle as subgraph. Then $G$ has Mahler measure at least $\tau_0$.
	\end{corollary}

	\subsubsection{Uncharged triangle-free graphs}

	Let $G$ be an uncharged triangle-free graph with four vertices. Then, up to equivalence, $G$ is one of the six graphs:
	
	
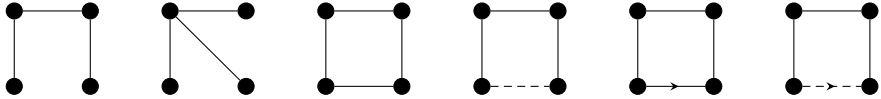
\[
	\begin{tikzpicture}
			\begin{scope}[scale=1]	
				\foreach \pos/\name/\sign/\charge in {{(0,0)/a/zc/{}}, {(1,0)/b/zc/{}}, {(0,-1)/c/zc/{}},{(1,-1)/d/zc/{}}}
					\node[\sign] (\name) at \pos {$\charge$}; 
				\foreach \edgetype/\source/ \dest /\weight in {pedge/a/b/{},pedge/b/d/{},pedge/a/c/{}}
				\path[\edgetype] (\source) -- (\dest);
			\end{scope}
			\end{tikzpicture}
	\hspace{2em}\begin{tikzpicture}
			\begin{scope}[scale=1]	
				\foreach \pos/\name/\sign/\charge in {{(0,0)/a/zc/{}}, {(1,0)/b/zc/{}}, {(0,-1)/c/zc/{}},{(1,-1)/d/zc/{}}}
					\node[\sign] (\name) at \pos {$\charge$}; 
				\foreach \edgetype/\source/ \dest /\weight in {pedge/a/b/{},pedge/a/d/{},pedge/a/c/{}}
				\path[\edgetype] (\source) -- (\dest);
			\end{scope}
			\end{tikzpicture}
	\hspace{2em}\begin{tikzpicture}
			\begin{scope}[scale=1]	
				\foreach \pos/\name/\sign/\charge in {{(0,0)/a/zc/{}}, {(1,0)/b/zc/{}}, {(0,-1)/c/zc/{}},{(1,-1)/d/zc/{}}}
					\node[\sign] (\name) at \pos {$\charge$}; 
				\foreach \edgetype/\source/ \dest /\weight in {pedge/a/b/{},pedge/b/d/{},pedge/a/c/{},pedge/c/d/{}}
				\path[\edgetype] (\source) -- (\dest);
			\end{scope}
			\end{tikzpicture}
	\hspace{2em}\begin{tikzpicture}
			\begin{scope}[scale=1]	
				\foreach \pos/\name/\sign/\charge in {{(0,0)/a/zc/{}}, {(1,0)/b/zc/{}}, {(0,-1)/c/zc/{}},{(1,-1)/d/zc/{}}}
					\node[\sign] (\name) at \pos {$\charge$}; 
				\foreach \edgetype/\source/ \dest /\weight in {pedge/a/b/{},pedge/b/d/{},pedge/a/c/{},nedge/c/d/{}}
				\path[\edgetype] (\source) -- (\dest);
			\end{scope}
			\end{tikzpicture}
	\hspace{2em}\begin{tikzpicture}
			\begin{scope}[scale=1]	
				\foreach \pos/\name/\sign/\charge in {{(0,0)/a/zc/{}}, {(1,0)/b/zc/{}}, {(0,-1)/c/zc/{}},{(1,-1)/d/zc/{}}}
					\node[\sign] (\name) at \pos {$\charge$}; 
				\foreach \edgetype/\source/ \dest /\weight in {pedge/a/b/{},pedge/b/d/{},pedge/a/c/{},wedge/c/d/{}}
				\path[\edgetype] (\source) -- (\dest);
			\end{scope}
			\end{tikzpicture}
	\hspace{2em}\begin{tikzpicture}
			\begin{scope}[scale=1]	
				\foreach \pos/\name/\sign/\charge in {{(0,0)/a/zc/{}}, {(1,0)/b/zc/{}}, {(0,-1)/c/zc/{}},{(1,-1)/d/zc/{}}}
					\node[\sign] (\name) at \pos {$\charge$}; 
				\foreach \edgetype/\source/ \dest /\weight in {pedge/a/b/{},pedge/b/d/{},pedge/a/c/{},wnedge/c/d/{}}
				\path[\edgetype] (\source) -- (\dest);
			\end{scope}
			\end{tikzpicture}
	\hspace{2em}
	\] which are all cyclotomic (for $\Z[i]$, the latter two are equivalent by conjugation). Any triangle-free uncharged minimal non-cyclotomic graph is therefore equivalent to an element of a $T_j$ grown from this seed set; we may more efficiently recover representatives by: taking $X=\{0\}$; filtering candidate columns for each matrix to discard those that would induce a triangle; and (for sufficiently large graphs) bounding the degree of vertices by Proposition \ref{wdegmax4}. Nonetheless, this remains a substantial computation, taking $3$ cpudays for $\Z[i]$ and $18$ for $\Z[\omega]$; Table \ref{searchresults2} summarises the results.

	\begin{table}[H]
	\begin{center}
	\begin{tabular}{|c|c|c|c|c|}
	\hline
	$j$ & \multicolumn{2}{|c|}{$|T_j|$} & \multicolumn{2}{|c|}{$\displaystyle\min_{A\in T_j}{M(A)}$} \\
 	& $\Z[i]$ & $\Z[\omega]$ &$\Z[i]$ & $\Z[\omega]$ \\ \hline
	$5$ & $3$ & $\leqslant 66$ & $2.081\ldots$ & $1.722\ldots$\\
	$6$ & $\leqslant 38$ & $\leqslant 37$ & $1.401\ldots$ & $1.31\ldots$\\
	$7$ &  $\leqslant 5$ & $\leqslant 38$ & $1.506\ldots$ & $1.50\ldots$\\
	$8$ & $\leqslant 51$ &$\leqslant  574$ & $1.35\ldots$ & $1.2806\ldots$\\
	$9$ & $\leqslant 124$ & $\leqslant 441$ & $1.2806\ldots$ & $1.2806\ldots$\\ \hline
	\end{tabular}
	\end{center}
	\caption{Least Mahler measures of small uncharged triangle-free graphs.}
	\label{searchresults2}
	\end{table}
	
	For $j=10$ the only minimal non-cyclotomic graphs recovered were $\Z$-graphs, and thus $M(A)\geqslant\tau_0$ for $A\in T_{10}$. We also note that for $\Z[\omega]$ the lower bound on the Mahler measure for $T_8$ is attained by a class of graphs that does not have a $\Z$-graph representative; nor does any $\Z$-graph have the same associated polynomial. However, there exist $\Z$-graphs with the same Mahler measure.

	From the search, we have:

	\begin{corollary}  Let $G$ be an uncharged triangle-free minimal non-cyclotomic graph of at most ten vertices. Then $G$ has Mahler measure at least $\tau_0$.
	\end{corollary}

	\subsubsection{Charged triangle-free graphs}

	Let $G$ be a charged triangle-free graph. Then, up to equivalence, $G$ induces as subgraph one of \[
	\begin{tikzpicture}[scale=1.3, auto]
		\begin{scope}	
			\foreach \pos/\name/\type/\charge in {{(0,0)/a/pc/1}, {(1,0)/b/zc/{}}}
				\node[\type] (\name) at \pos {$\charge$}; 
			\foreach \edgetype/\source/ \dest / \weight in {pedge/a/b/{}}
			\path[\edgetype] (\source) -- node[weight] {$\weight$} (\dest);
		\end{scope}
		\end{tikzpicture}\hspace{2em}
	\begin{tikzpicture}[scale=1.3, auto]
		\begin{scope}	
			\foreach \pos/\name/\type/\charge in {{(0,0)/a/pc/1}, {(1,0)/b/pc/1}}
				\node[\type] (\name) at \pos {$\charge$}; 
			\foreach \edgetype/\source/ \dest / \weight in {pedge/a/b/{}}
			\path[\edgetype] (\source) -- node[weight] {$\weight$} (\dest);
		\end{scope}
		\end{tikzpicture}\hspace{2em}
	\begin{tikzpicture}[scale=1.3, auto]
		\begin{scope}	
			\foreach \pos/\name/\type/\charge in {{(0,0)/a/pc/1}, {(1,0)/b/nc/1}}
				\node[\type] (\name) at \pos {$\charge$}; 
			\foreach \edgetype/\source/ \dest / \weight in {pedge/a/b/{}}
			\path[\edgetype] (\source) -- node[weight] {$\weight$} (\dest);
		\end{scope}
		\end{tikzpicture}\hspace{2em}
	\] which are all cyclotomic. Any triangle-free uncharged minimal non-cyclotomic graph is therefore equivalent to an element of a $T_j$ grown from this seed set; we may more efficiently recover representatives by filtering candidate columns for each matrix to discard those that would induce a triangle; and (for sufficiently large graphs) bounding the degree of vertices by Proposition \ref{wdegmax4}. Again, this requires a significant level of computation ($2$/$4.5$ cpudays for $\Z[i]$/$\Z[\omega]$ respectively); Table~\ref{searchresults3} summarises the results.

	\begin{table}[H]
	\begin{center}
	\begin{tabular}{|c|c|c|c|c|l|l|}
	\hline
	$i$ & \multicolumn{2}{|c|}{$|T_j|$} & \multicolumn{2}{|c|}{$\displaystyle\min_{A\in T_j}{M(A)}$}\\
 	& $\Z[i]$ & $\Z[\omega]$ &$\Z[i]$ & $\Z[\omega]$\\\hline
	$3$ & $6$ & $6$ & $1.506\ldots$ & $1.506\ldots$\\
	$4$ & $22$ & $26$ & $1.2806\ldots$ & $1.2806\ldots$\\
	$5$ & $27$ & $\leqslant 124$ & $\tau_0$ & $\tau_0$\\
	$6$ & $\leqslant 15$ & $17$ & $1.240\ldots$ & $1.240\ldots$\\
	$7$ &  $1$ & $2$ & $1.216\ldots$ & $1.216\ldots$\\
	$8$ & $1$ & $2$ & $1.200\ldots$ & $1.200\ldots$\\
	$9$ & $1$ & $1$ & $1.188\ldots$ & $1.188\ldots$\\
	$10$ & $1$ & $1$ & $1.254\ldots$ & $1.254\ldots$\\\hline
	\end{tabular}
	\end{center}
	\caption{Least Mahler measures of charged triangle-free graphs.}
	\label{searchresults3}
	\end{table}

	We note that for $\Z[\omega]$ there exist classes of graphs with small Mahler measure that do not have a $\Z$-graph representative. However, there exist $\Z$-graphs with the same Mahler measure.

	From the search, we have:

	\begin{corollary}  Let $G$ be an uncharged triangle-free minimal non-cyclotomic graph of at most ten vertices. Then $G$ has Mahler measure at least $\tau_0$.
	\end{corollary}

	Combining the results of this section so far, we therefore conclude:

	\begin{theorem} Let $G$ be a minimal non-cyclotomic graph of at most ten vertices. Then $G$ has Mahler measure at least $\tau_0$.
	\end{theorem}

	\subsection{Nonexistence of supersporadic minimal non-cyclotomic graphs with eleven or more vertices}

	\subsubsection{Excluded subgraphs}
	We describe a graph $G$ (and its equivalents) as being of type-I if it is non-cyclotomic; clearly, no larger non-cyclotomic graph can induce $G$ as a subgraph and still be minimal. Type-I graphs containing weight-$2$ edges are given in Figure \ref{typeIfigure}; we note that the graph $X_4$ given in Figure \ref{typeIfigureX4} is also type-I.

	Further, there exist graphs which, whilst cyclotomic, are contained in only finitely many cyclotomic supergraphs; we call these type-II. If the largest supergraph of a type-II graph $G$ has $n$ vertices, then any minimal non-cyclotomic graph containing $G$ has at most $n+1$ vertices; thus in seeking minimal non-cyclotomics of at least eleven vertices, we may exclude any subgraph that only occurs in cyclotomic graphs of nine or fewer vertices.

	\begin{proposition}
	Triangles with weight-$1$ edges and no, one, or three charged vertices are type-II, and cannot be subgraphs of any minimal non-cyclotomic graph with eleven or more vertices.
	\end{proposition}

	\begin{proposition}
	The graphs with weight-$2$ edges given in Figure \ref{typeIIfigure} are type-II, and cannot be subgraphs of any minimal non-cyclotomic graph with eleven or more vertices.
	\end{proposition}

	Finally, by Proposition \ref{wdegmax4} any vertex in a minimal non-cyclotomic graph with eleven or more vertices must have degree at most four, so we may use bounded vectors and discard any vector/charge pair that would induce a higher degree vertex.

	\begin{figure}[H]
	\[
	\begin{tikzpicture}[scale=0.8, auto]
		\begin{scope}	
			\foreach \pos/\name/\type/\charge in {{(210:1)/a/zc/{}}, {(330:1)/b/zc/{}},{(90:0.8)/c/zc/{}}}
				\node[\type] (\name) at \pos {$\charge$}; 
			\foreach \edgetype/\source/ \dest / \weight in {wwedge/a/b/{},wwedge/a/c/{},pedge/b/c/{}}
			\path[\edgetype] (\source) -- node[weight] {$\weight$} (\dest);
		\end{scope}
		\end{tikzpicture}\hspace{2em}		
	\begin{tikzpicture}[scale=0.8, auto]
		\begin{scope}	
			\foreach \pos/\name/\type/\charge in {{(210:1)/a/zc/{}}, {(330:1)/b/zc/{}},{(90:0.8)/c/zc/{}}}
				\node[\type] (\name) at \pos {$\charge$}; 
			\foreach \edgetype/\source/ \dest / \weight in {wwedge/a/b/{},wwedge/c/a/{},pedge/b/c/{}}
			\path[\edgetype] (\source) -- node[weight] {$\weight$} (\dest);
		\end{scope}
		\end{tikzpicture}\hspace{2em}
		\begin{tikzpicture}[scale=0.8, auto]
			\begin{scope}	
				\foreach \pos/\name/\type/\charge in {{(210:1)/a/zc/{}}, {(330:1)/b/zc/{}},{(90:0.8)/c/zc/{}}}
					\node[\type] (\name) at \pos {$\charge$}; 
				\foreach \edgetype/\source/ \dest / \weight in {wwedge/a/b/{},pedge/a/c/{},pedge/b/c/{}}
				\path[\edgetype] (\source) -- node[weight] {$\weight$} (\dest);
			\end{scope}
		\end{tikzpicture}	
	\]
	\[
	\begin{tikzpicture}[scale=1, auto]
		\begin{scope}	
			\foreach \pos/\name/\type/\charge in {{(0,0)/a/pc/1}, {(1,0)/b/pc/1}}
				\node[\type] (\name) at \pos {$\charge$}; 
			\foreach \edgetype/\source/ \dest / \weight in {wwedge/a/b/{}}
			\path[\edgetype] (\source) -- node[weight] {$\weight$} (\dest);
		\end{scope}
		\end{tikzpicture}\hspace{2em}
	\begin{tikzpicture}[scale=1, auto]
		\begin{scope}	
			\foreach \pos/\name/\type/\charge in {{(0,0)/a/zc/{}}, {(1,0)/b/zc/{}},{(2,0)/c/zc/{}},{(3,0)/d/pc/1}}
				\node[\type] (\name) at \pos {$\charge$}; 
			\foreach \edgetype/\source/ \dest / \weight in {wwedge/a/b/{},wwedge/b/c/{},pedge/c/d/{}}
			\path[\edgetype] (\source) -- node[weight] {$\weight$} (\dest);
		\end{scope}
		\end{tikzpicture}\hspace{2em}		
	\begin{tikzpicture}[scale=1, auto]
		\begin{scope}	
			\foreach \pos/\name/\type/\charge in {{(0,0)/a/zc/{}}, {(1,0)/b/zc/{}},{(2,0)/c/zc/{}},{(3,0)/d/zc/{}}}
				\node[\type] (\name) at \pos {$\charge$}; 
			\foreach \edgetype/\source/ \dest / \weight in {wwedge/a/b/{},wwedge/b/c/{},pedge/c/d/{}}
			\path[\edgetype] (\source) -- node[weight] {$\weight$} (\dest);
		\end{scope}
		\end{tikzpicture}		
	\]
	\[
	\begin{tikzpicture}[scale=1, auto]
		\begin{scope}	
			\foreach \pos/\name/\type/\charge in {{(0,0)/a/zc/{}}, {(1,0)/b/zc/{}},{(0,-1)/c/zc/{}},{(1,-1)/d/zc/{}}}
				\node[\type] (\name) at \pos {$\charge$}; 
			\foreach \edgetype/\source/ \dest / \weight in {wwedge/a/b/{},pedge/a/c/{},pedge/c/d/{},pedge/d/b/{}}
			\path[\edgetype] (\source) -- node[weight] {$\weight$} (\dest);
		\end{scope}
		\end{tikzpicture}\hspace{2em}
	\begin{tikzpicture}[scale=1, auto]
		\begin{scope}	
			\foreach \pos/\name/\type/\charge in {{(0,0)/a/zc/{}}, {(1,0)/b/zc/{}},{(0,-1)/c/zc/{}},{(1,-1)/d/zc/{}}}
				\node[\type] (\name) at \pos {$\charge$}; 
			\foreach \edgetype/\source/ \dest / \weight in {wwedge/a/b/{},pedge/a/c/{},wwedge/c/d/{},pedge/d/b/{}}
			\path[\edgetype] (\source) -- node[weight] {$\weight$} (\dest);
		\end{scope}
		\end{tikzpicture} \hspace{2em}
		\begin{tikzpicture}[scale=1, auto]
			\begin{scope}	
				\foreach \pos/\name/\type/\charge in {{(0,0)/a/zc/{}}, {(1,0)/b/zc/{}},{(0,-1)/c/zc/{}},{(1,-1)/d/zc/{}}}
					\node[\type] (\name) at \pos {$\charge$}; 
				\foreach \edgetype/\source/ \dest / \weight in {wwedge/a/b/{},pedge/a/c/{},wwedge/d/c/{},pedge/d/b/{}}
				\path[\edgetype] (\source) -- (\dest);
			\end{scope}
			\end{tikzpicture}		
	\]
	\caption{Type-I subgraphs with weight-$2$ edges.}
	\label{typeIfigure}
	\end{figure}
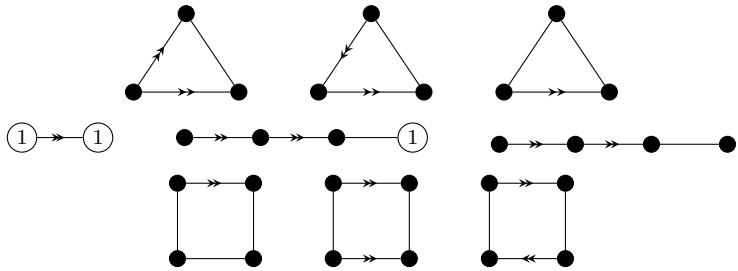

	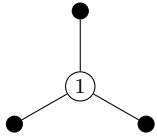
\begin{figure}	
	\[
	\begin{tikzpicture}[scale=1, auto]
		\begin{scope}	
			\foreach \pos/\name/\type/\charge in {{(210:1)/a/zc/{}}, {(330:1)/b/zc/{}},{(90:1)/c/zc/{}},{(0,0)/d/pc/1}}
				\node[\type] (\name) at \pos {$\charge$}; 
			\foreach \edgetype/\source/ \dest / \weight in {pedge/a/d/{},pedge/b/d/{},pedge/c/d/{}}
			\path[\edgetype] (\source) -- node[weight] {$\weight$} (\dest);
		\end{scope}
		\end{tikzpicture}
	\]
	\caption{The type-I subgraph $X_4$.}
	\label{typeIfigureX4}
	\end{figure}

	\begin{figure}[H]
	\[
	\begin{tikzpicture}[yscale=0.8, auto]
		\begin{scope}	
			\foreach \pos/\name/\type/\charge in {{(0,0)/a/pc/1}, {(1,0)/b/nc/1},{(2,0)/a2/pc/1},{(3,0)/b2/zc/{}},{(0,1)/a3/zc/{}},{(1,1)/b3/zc/{}},{(2,1)/c3/zc/{}},{(3,1)/d3/zc/{}}}
				\node[\type] (\name) at \pos {$\charge$}; 
			\foreach \edgetype/\source/ \dest / \weight in {wwedge/a/b/{},wwedge/a2/b2/{},wwedge/b3/c3/{},pedge/a3/b3/{},pedge/c3/d3/{}}
			\path[\edgetype] (\source) -- node[weight] {$\weight$} (\dest);
		\end{scope}
	\end{tikzpicture}\hspace{2em}		
	\begin{tikzpicture}[scale=1, auto]
		\begin{scope}	
			\foreach \pos/\name/\type/\charge in {{(0,0)/a/zc/{}}, {(1,0)/b/zc/{}},{(0,-1)/c/zc/{}},{(1,-1)/d/zc/{}}}
				\node[\type] (\name) at \pos {$\charge$}; 
			\foreach \edgetype/\source/ \dest / \weight in {wwedge/a/b/{},pedge/a/c/{},wwnedge/c/d/{},pedge/d/b/{}}
			\path[\edgetype] (\source) -- node[weight] {$\weight$} (\dest);
		\end{scope}
	\end{tikzpicture}\hspace{2em}	
	\begin{tikzpicture}[scale=1, auto]
		\begin{scope}	
			\foreach \pos/\name/\type/\charge in {{(0,0)/a/zc/{}}, {(1,0)/b/zc/{}},{(0,-1)/c/zc/{}},{(1,-1)/d/zc/{}}}
				\node[\type] (\name) at \pos {$\charge$}; 
			\foreach \edgetype/\source/ \dest / \weight in {wwedge/a/b/{},pedge/a/c/{},nedge/c/d/{},pedge/d/b/{}}
			\path[\edgetype] (\source) -- node[weight] {$\weight$} (\dest);
		\end{scope}
	\end{tikzpicture}
	\]
	\caption{Type-II subgraphs with weight-$2$ edges.}
	\label{typeIIfigure}
	\end{figure}
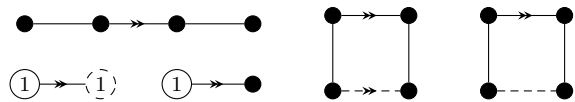
	
	\subsubsection{Included subgraphs}
	By Proposition \ref{pro:unchargedRed}, If $G$ has eleven or more vertices it must be supersporadic, and must therefore induce a subgraph equivalent to $YA_1$, $YA_2$ or $YA_3$ (see Figure \ref{YA_ifigure}). Growing from these seeds gives a smaller set of ten-vertex candidates than considering all possible subgraphs of $S_{14}$ or $S_{16}$. Due to duplication of classes and the cost of equivalence testing, it is not worth growing beyond this point, whilst testing for a $YA_j$ subgraph is also prohibitively expensive.
	
	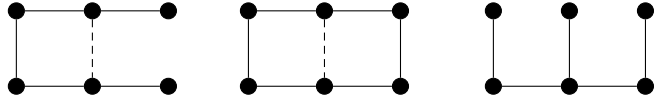
\begin{figure}[H]
	\[
	\begin{tikzpicture}[scale=1, auto]
		\begin{scope}	
			\foreach \pos/\name/\type/\charge in {{(0,0)/a/zc/{}}, {(1,0)/b/zc/{}},{(0,-1)/c/zc/{}},{(1,-1)/d/zc/{}},{(2,0)/e/zc/{}},{(2,-1)/f/zc/{}}}
				\node[\type] (\name) at \pos {$\charge$}; 
			\foreach \edgetype/\source/ \dest / \weight in {pedge/a/b/{},pedge/b/e/{},pedge/a/c/{},pedge/c/d/{},pedge/d/f/{},nedge/b/d/{}}
			\path[\edgetype] (\source) -- node[weight] {$\weight$} (\dest);
		\end{scope}
	\end{tikzpicture}\hspace{2em}
	\begin{tikzpicture}[scale=1, auto]
		\begin{scope}	
			\foreach \pos/\name/\type/\charge in {{(0,0)/a/zc/{}}, {(1,0)/b/zc/{}},{(0,-1)/c/zc/{}},{(1,-1)/d/zc/{}},{(2,0)/e/zc/{}},{(2,-1)/f/zc/{}}}
				\node[\type] (\name) at \pos {$\charge$}; 
			\foreach \edgetype/\source/ \dest / \weight in {pedge/a/b/{},pedge/b/e/{},pedge/a/c/{},pedge/c/d/{},pedge/d/f/{},nedge/b/d/{},pedge/e/f/{}}
			\path[\edgetype] (\source) -- node[weight] {$\weight$} (\dest);
		\end{scope}
	\end{tikzpicture}\hspace{2em}
	\begin{tikzpicture}[scale=1, auto]
		\begin{scope}	
			\foreach \pos/\name/\type/\charge in {{(0,0)/a/zc/{}}, {(1,0)/b/zc/{}},{(0,-1)/c/zc/{}},{(1,-1)/d/zc/{}},{(2,0)/e/zc/{}},{(2,-1)/f/zc/{}}}
				\node[\type] (\name) at \pos {$\charge$}; 
			\foreach \edgetype/\source/ \dest / \weight in {pedge/e/f/{},pedge/a/c/{},pedge/c/d/{},pedge/d/f/{},pedge/b/d/{}}
			\path[\edgetype] (\source) -- node[weight] {$\weight$} (\dest);
		\end{scope}
	\end{tikzpicture}
	\]
	\caption{The subgraphs $YA_1$, $YA_2$, $YA_3$.}
	\label{YA_ifigure}
	\end{figure}
	
	\subsubsection{Supergraphs of $S_{10}$ or $S_{12}$ subgraphs}
	Due to the bound on the degree of each vertex, we need only test the eleven- and ten-vertex subgraphs of $S_{12}$. Testing by $\mathfrak C_n^4(\mathcal{L})$ and $\mathfrak C_n^3(\mathcal{L})$ for uncharged and charged extensions respectively could be optimised by excluded subgraph arguments, but the search space is small enough for this to not matter in practice; the test can be completed in $16.5$ cpuhours.

	\subsubsection{Charged supergraphs of $S_{14}$ or $S_{16}$ subgraphs}
	Let $H$ be an induced subgraph of $S_{14}$ or $S_{16}$, and $G$ the supergraph formed by adding a charged vertex $x$ with addition vector $\mathbf{c}$. For any chance of $G$ to be minimally non-cyclotomic, we may force $\mathbf{c}\in \mathfrak C_n^2(\mathcal{L}_1\cup\{0\})$ with at least one complex edge, as follows:
	\begin{itemize}
	\item as $x$ is charged, we force $c\in \mathfrak C_n^3(\mathcal{L})$;
	\item to exclude $\begin{tikzpicture}[scale=1, auto]
		\begin{scope}	
			\foreach \pos/\name/\type/\charge in {{(0,0)/a/pc/1}, {(1,0)/b/pc/1}}
				\node[\type] (\name) at \pos {$\charge$}; 
			\foreach \edgetype/\source/ \dest / \weight in {wwedge/a/b/{}}
			\path[\edgetype] (\source) -- node[weight] {$\weight$} (\dest);
		\end{scope}
		\end{tikzpicture}$, $\begin{tikzpicture}[scale=1, auto]
		\begin{scope}	
			\foreach \pos/\name/\type/\charge in {{(0,0)/a/pc/1}, {(1,0)/b/nc/1}}
				\node[\type] (\name) at \pos {$\charge$}; 
			\foreach \edgetype/\source/ \dest / \weight in {wwedge/a/b/{}}
			\path[\edgetype] (\source) -- node[weight] {$\weight$} (\dest);
		\end{scope}
		\end{tikzpicture}$ and $\begin{tikzpicture}[scale=1, auto]
		\begin{scope}	
			\foreach \pos/\name/\type/\charge in {{(0,0)/a/pc/1}, {(1,0)/b/zc/{}}}
				\node[\type] (\name) at \pos {$\charge$}; 
			\foreach \edgetype/\source/ \dest / \weight in {wwedge/a/b/{}}
			\path[\edgetype] (\source) -- node[weight] {$\weight$} (\dest);
		\end{scope}
		\end{tikzpicture}$ , we restrict to $\mathfrak C_n^3(\mathcal{L}_1\cup\{0\})$;
	\item if $x$ has degree $4$, then it either induces $X_{4}$ or a triangle with a single charge. So we restrict to  $\mathfrak C_n^2(\mathcal{L}_1\cup\{0\})$.
	\end{itemize}

	Then, on a per-graph basis, we discard any columns which would exceed the degree bound at any vertex, or induce a triangle with $x$ as its only charged vertex. Total compute time is then $28$ cpuhours for $\Z[i]$, $48$ cpuhours for $\Z[\omega]$.		

	\subsubsection{Supergraphs with weight-$2$ edges}
	Let $H$ be an induced subgraph of $S_{14}$ or $S_{16}$, and $G$ the supergraph formed by adding a vertex $x$ with addition vector $\mathbf{c}$, where at least one entry of $\mathbf{c}$ has weight-$2$. We note that $x$ must be uncharged. 

	If $H$ is connected, then if $c$ had both weight-$1$ and weight-$2$ entries, $G$ would induce a subgraph with an isolated weight-$2$ edge, which is either type I or II. So for such $H$, we need only consider $\mathfrak C_n^4(\mathcal{L}_2\cup\{0\})$; for the few disconnected examples, we consider all $\mathbf{c}\in \mathfrak C_n^4(\mathcal{L})$ with a weight-$2$ edge, subject to local weight bounds, and check $G$ for connectedness. This requires $8.5$ cpuhours in total.

	\subsubsection{Remaining cases}
	This leaves only uncharged supergraphs with all edges of weight $1$; by the weight bound, $\mathbf{c}\in \mathfrak C_n^4(\mathcal{L}_1\cup\{0\})$. We may assume $\mathbf{c}$ has at least one complex entry, else we have a $\Z$-graph which is necessarily not minimally non-cyclotomic. Further, for each $H$ we check weight bounds at each vertex, and discard any vector which would induce an uncharged triangle. This represents the bulk of the finite search, requiring $19$ cpudays for $\Z[i]$ and $47$ cpudays for $\Z[\omega]$.


\bibliographystyle{amsalpha}
\bibliography{bib}
\end{document}